\documentclass{amsart}
\usepackage{amssymb,amsmath}
\usepackage[latin1]{inputenc}
\newcommand{\fint}{\int \! \! \! \! \!-}
\newcommand{\Sphere}{\mathbb S}
\newcommand{\upto}{\nearrow}
\newcommand{\cut}{\mbox{\rm  cut }} 

\newcommand{\dist}{\mbox{\rm dist}}
\newcommand{\fall}{ \ \ \forall \ \ }

\newcommand{\itoinfty}{\stackrel{ i \to \infty}{\longrightarrow}}
\newcommand{\GHd}{d_{GH}}
\newcommand{\GHlim}{ {\rm GH} \lim }

\newcommand{\Riem}{{\rm Riem}}

\newcommand{\vol}{{\rm vol}}

\newcommand{\inj}{{\rm inj}}

\newcommand{\diam}{\mbox{\rm diam }}

\newcommand{\of}{\rm{o}}

\newcommand{\curlO}{ {\mathcal O} }

\newcommand{\curlT}{ {\mathcal T} }
\newcommand{\curlTi}{ {\mathcal T_{\infty} }}
\newcommand{\curlR}{ {\mathcal R} }

\newcommand{\Ricci}{ {\rm Ricci}}

\newcommand{\at}{ {\rm (${\rm a}_t$)} }
\newcommand{\bt}{ {\rm (${\rm b}_t$)} }
\newcommand{\ct}{ {\rm(${\rm c}_t$)} }
\newcommand{\dt}{ {\rm(${\rm d}_t$)} }

\newcommand{\ep}{\varepsilon}

\newcommand{\grad}{\nabla}

\newcommand{\la}{\lambda}

\newcommand{\si}{\sigma}\newcommand{\al}{\alpha}
\newcommand{\intersect}{\cap}

\newcommand{\be}{\beta}
\newcommand{\ga}{\gamma}
\newcommand{\tvz}{\tilde{v_0} }

\newcommand{\parti}[2]{\frac{\partial #1 } {\partial #2} }
\newcommand{\lap}{\Delta}

\newcommand{\upi}{{}^{^i}\!}
\newcommand{\on}{\over}

\newcommand{\ti}{\tilde}

\newcommand{\upig}{ \upi g }
\newcommand{\upghi}{ \upi {\hat g} }

\newcommand{\upl}{^{^l}\!}
\newcommand{\upg}{^{^g}\!}
\newcommand{\upgi}{^{^{^i \!g}}\!}

\newcommand{\partt}{ {\partial \on {\partial t}} }

\newcommand{\Sc}{ { \rm R }}

\newcommand{\timess}{\times}

\newcommand{\nonum}{ \nonumber }

\newcommand{\br}{ {\begin{rema}}}
\newcommand{\bl}{ {\begin{lemm}} }
\newcommand{\er} { {\end{rema}}}
\newcommand{\el}{{\end{lemm}} }
\newcommand{\bc}{  {\begin{coro}} }
\newcommand{\ec}{  {\end{coro}} }

\newcommand{\plusminus}{ \plusminus}

\newcommand{\de}{\delta}

\newcommand{\up}[1]{{^{^{#1}}\!\!}}

\newcommand{\Sp}{\mathbb{S}}
\newcommand{\N}{\mathbb{N}}
\newcommand{\R}{\mathbb{R}}

\newcommand{\Haus}{\mathcal{H}}

\usepackage{amsthm} %braucht man fuer swapnumber,... 
\usepackage{amsmath}% enth"alt amstext,amsbsy,amsopn siehe script mehr spalten als bei eqnarray 
\usepackage{amssymb} % f"ur mathematische Symbole der ams-Fonts
\usepackage{latexsym} 
\usepackage{array} 
\usepackage[T1]{fontenc} % 
\usepackage{fancyhdr} % seitenstil (strich,...)%
\usepackage{xspace} % 
\usepackage[dvips]{epsfig} % grafiken einbeziehen 
\usepackage{graphicx} % grafiken einbeziehen
 \usepackage[T1]{fontenc}
    \usepackage{ae}
    \usepackage{aecompl}

\newcommand{\RP}{{\mathbb R \mathbb P}}
\newtheorem{defi}{Definition}[section]

\newtheorem{prob}[defi]{Problem}
\newtheorem{rema}[defi]{Remark}

\newtheorem{exam}[defi]{Example}
\newtheorem{theo}[defi]{Theorem}
\newtheorem{lemm}[defi]{Lemma}

\newtheorem{coro}[defi]{Corollary}

\newtheorem{conj}[defi]{Conjecture}
\newtheorem{prop}[defi]{Proposition}

\makeatother
\title[Ricci flow of non-collapsed three manifolds with $\Ricci \geq -k$]{ Ricci flow of non-collapsed three manifolds 
\\ whose Ricci curvature is bounded from below}
\author[Miles Simon]{Miles Simon\\Universit\"at Freiburg${}^{{}^*}$}
\thanks{* part of this work was completed during the author's stay at
 Universit\"at M\"unster in the semester 2008/09. This work was
partially supported by SFB/Transregio 71}
\begin{document}
\maketitle
%\vskip  0.05 true in
%\vspace{-0.8cm}
 \numberwithin{equation}{section}
\numberwithin{defi}{section}
%\textwidth100mm 
%\center{\ \ Universit\"at  Freiburg\footnote{testy} } 
%\textwidth220mm 
%\vspace{0.9cm}

\begin{abstract}
We consider complete (possibly non-compact) three dimensional Riemannian manifolds $(M,g)$ such that: 
a) $(M,g)$ is non-collapsed (i.e. the volume of an arbitrary ball of radius one is bounded from below
by $v>0$ ) b) the Ricci curvature of $(M,g)$ is bounded from below by $k$, 
c) the geometry at infinity of $(M,g)$ is not too extreme (or $(M,g)$ is compact). 
Given such initial data $(M,g)$ we show that a Ricci flow exists for a short time interval $[0,T)$,
where $T =T(v,k)>0$. 
This enables us to construct a Ricci flow of any (possibly singular) metric space $(X,d)$ which
arises as a Gromov-Hausdorff limit of a sequence of 3-manifolds which satisfy a), b) and c) uniformly.
As a corollary we show that such an $X$ must be a manifold.
This shows that the conjecture of M.Anderson-J.Cheeger-T.Colding-G.Tian
is correct in dimension three.

\end{abstract}

\vskip  0.07 true in
\section{Introduction and statement of results}
\vskip  0.1 true in

A smooth family of metrics $(M,g(t))_{t \in [0,T)}$ is a solution to the Ricci flow if
\begin{eqnarray}
&&\partt g(t)= - 2 \Ricci(g(t)) \ \forall \ t \in [0,T). \cr
\end{eqnarray}
We say that this solution has initial value $g_0$ if $g(\cdot,0) = g_0(\cdot)$.
The Ricci flow was introduced by R. Hamilton in \cite{Ha82} and has led to many
new results in differential geometry and topology
: see for example \cite{Pe1},\cite{Pe2}, \cite{BoWi}, \cite{ScBr},\cite{Ng},\cite{CCCY},\cite{Hu},\cite{Ma}
\cite{ChZh}.
For very good expositions of the papers of G.Perelman (\cite{Pe1,Pe2}) and parts 
thereof see
\cite{CaZh}, \cite{TiMo1,TiMo2}, \cite{KlLo} and
\cite{Re} and \cite{ChNi}.

In this paper we define a Ricci flow
for a class of possibly singular metric spaces, elements of
which arise as Gromov-Hausdorff limits
of sequences of complete, non-collapsed manifolds with Ricci curvature bounded from below.

More specifically, we 
consider the class of smooth, complete Riemannian manifolds $(M,g)$ 
which satisfy
%\begin{defi}\label{curlTdef}
%\hfill\break
\begin{itemize}
\item[(a)]  $\Ricci(g) \geq k$
\item[(b)] $ \vol({\upg B}_1(x)) \geq v_0 > 0$ for all $x \in M$
\end{itemize}
%\end{defi}

It is well known, see \cite{Gr}, that every sequence of smooth
Riemannian manifolds satisfying (a) contains a subsequence which converges
with respect to the Gromov-Hausdorff distance 
to a possibly singular metric space $(M,d)$
(see \cite{BuBuIv} for a definition of Gromov-Hausdorff distance: this distance 
is a weak measure of how close metric spaces are to being isometric).
With the expression 'possibly singular' we mean two things:
\begin{itemize}
\item it is possible that the limiting space $(M,d)$ is no longer 
a manifold (see Example \ref{examanderson} below) and 
\item it is possible that the resulting
metric $d$ is not smooth, even if $M$ is a manifold
(see Example \ref{examcone} below).
\end{itemize}

\begin{exam}\label{examanderson}(M.Anderson)
{\rm This example is from M.Anderson (see section 3 of \cite{AnEx}).
In the paper \cite{EgHa} T.Eguchi and A.Hanson construct a 
four dimensional Riemannian manifold $(M^4,h)$
where $M = TS^2$ and $\Ricci(h) = 0$ everywhere.
Asymptotically (far away from some base point) the Riemmanian manifold
looks like a cone over $\RP^3$. More explicitly:
if we rescale the metric, $(M_i,h_i):=(M,{1 \on i}h)$, then
$\Ricci(h_i) = 0$, $\vol(B_1(x),h_i) \geq v_0$ for all $i \in \N$ 
and $ (M_i,d(h_i)) \to (N,l)$ as $i \to \infty$ where
$N = (\R^+_0\times \RP^3)/ (\{0\}  \times \RP^3)$ with the quotient topology, where
$l((r,x),(s,y)):= \sqrt{r^2 + s^2 - 2rs \cos(\ga(x,y)) }$
for all $r, s \in \R^+_0$ and all $x,y \in \RP^3$
and $\gamma:\RP^3 \times \RP^3 \to \R^+_0 $ is the standard
distance on $\RP^3$.
In particular, $N$ is not a manifold.}
\end{exam}

\begin{exam}\label{examcone}
{\rm Let $(M^n,h)$ be a non-negatively curved {\it smoothed out} cone over $S^{n-1}$. 
That is, we give $M^n = \R^n =  (\R^+_0\times \Sp^{n-1})/ (\{0\} \times \Sp^{n-1})$
a smooth metric $h$ such that $\sec(h) \geq 0$ everywhere and 
$h(r,\al) = dr^2 \oplus cr^2 \ga(\al)$ for $r \geq 1$ and some constant $0<c < 1$,
$\ga$ the standard metric on $S^{n-1}$.
Let $(M_i,h_i):= (M,{1 \on i}h)$. Clearly $\vol(B_1(x),h_i) \geq v_0$ for
some $v_0 >0$ and all $i \in \N$, all $x \in M_i$.
Also, $(M_i,d(h_i)) \to (M,l)$ 
where 
$l((r,x),(s,y)):= \sqrt{r^2 + s^2 - 2rs \cos(\sqrt{c}\psi(x,y)) }$
for all $r, s \in \R^+_0$ and all $x,y \in \Sp^{n-1}$
and $\psi:\Sp^{n-1} \times \Sp^{n-1} \to \R^+_0 $ is the standard
distance on $\Sp^{n-1}$.
The distance $l:M \times M \to \R$
is then continuous, but not differentiable everywhere.
For example: if $p = (1,0, \ldots,0),q(x) = (0,x,0,\ldots,0) \in \R^n$ in euclidean coordinates, then
$f(x):= l^2( p, q(x)) =  1 + |x|^2 - 2|x| \cos(\sqrt{c} (\pi/2))$ is continuous
in $x=0$ but not differentiable there 
(since $\sqrt{c} < 1 \implies \cos(\sqrt{c} (\pi/2)) \neq 0$).}
\end{exam}

\begin{rema}
Any metric space $(M,d)$ which arises as the GH limit of a sequence of
two dimensional 
Riemannian manifolds satisfying (a) and (b) is itself a manifold. 
\end{rema}
This is because: in dimension two
$\Ricci \geq -k^2 \implies \sec \geq -2k^2$.
Then a Theorem of G.Perelman says that $(M,d)$ is a manifold:
see \cite{Ka}.

So we see that in dimension two any metric space $(M,d)$ which arises as the GH limit
of a sequence of Riemannian manifolds satisfying (a),(b) must be a manifold, and
in dimension four, there are examples where such $(M,d)$´s are not
manifolds.
It is a conjecture of M.Anderson-J.Cheeger-T.Colding-G.Tian (see the introduction
of \cite{ChCo}), that
\begin{conj}(M.Anderson-J.Cheeger-T.Colding-G.Tian ) 
Any metric space $(M,d)$ which arises as the GH limit
of a sequence of three dimensional Riemannian 
manifolds satisfying
(a) and (b) is itself a manifold. 
\end{conj}

In this paper we obtain as a consequence of one of our main theorems
(Theorem \ref{realmaintheo2} in this paper) that this conjecture is correct, if each of the
manifolds occuring in the sequence is compact or we demand that the geometry at infinity is
controlled in a certain sense (see condition (c) and $\ti{\rm c}$ below).
That is we will assume that each of the manifolds 
$(M,g)$ occuring in the sequence satisfy additionally   
\begin{itemize}
\item[(c)] $\sup_M | \Riem(g) | < \infty$ ({\it bounded curvature})
\end{itemize}
or
\begin{itemize}
\item[(${\ti{\rm{c}}}$ )] Let $f:M \to \R$ be the exponential
function composed with its self $m$-times, and $\rho:M \to \R_0^+$
the distance function from a fixed base point $b$, $\rho(x) := \dist(x,b)$.
We assume that 
\begin{itemize}

\item[ (${\ti{\rm{c_1}}}$ )]
$\rho:M-B_R(b) \to \R$ is smooth for some $R>0$,
and $k-concave$ there, that is
$$ \grad^2 \rho \leq k,$$ on $(M-B_R(b))$ and
\item[ (${\ti{\rm{c_2}}}$)]
$$\lim_{r \to \infty} (\sup_{x \in B_r(b)}|\Riem(x)|/f(r))  = 0.$$
\end{itemize}
\end{itemize}

\begin{rema}
Note that condition (c) is trivially satisfied if $M$ is compact.
\end{rema}
\begin{rema}
Assume ($\ti{\rm c_2}$) is satisfied for some $m \in N$, and that
the sectional curvatures of $(M,g)$ are larger
than $-l$ on $M$  and that $\cut(b) \cap (M - B_R(b)) = \emptyset$ for some $R> 0$.
Then condition ($\ti{\rm c_1}$) is satisfied for some $k = k(n,l)$ and some larger 
$m$ (depending on the initial $m$), as as one sees using the hessian comparison principle (see for example chapter 1 of \cite{ScYa}).
\end{rema}

Under these restrictions, we obtain that the conjecture of Anderson-Cheeger-Colding-Tian is correct.
That is, we prove:
\begin{theo}\label{chcoconj}
Let $(X,d_X)$ be a metric space arising as the GH limit of a seqeunce of three dimensional 
Riemannian manifolds $(M_i,g_i)$, $i \in \N$ each of which satisfies
$(a),(b)$ and $(c)$ or each of which satisfies $(a),(b)$ and $(\ti c)$.
Then $X$ is a three dimensional manifold.
If furthermore each of the $(M_i,g_i)$ has diameter bounded above by a uniform
constant $d_0 < \infty$,
then $M_i $ is diffeomorphic to $X$ for all $i$ sufficiently large.
\end{theo}

\begin{rema}
In the case that all manifolds in the sequence above satisfy
a two sided Ricci curvature bound,
$|\Ricci| \leq k^2$, a bound on the integral of the curvature tensor
$\int_M |\Riem|^{3/2} \leq D$  and (b) is satisfied, M.Anderson also proved
that the limit space $X$ is a manifold: see Corollary 2.8 in \cite{AnCon}.
Later, Cheeger-Colding-Tian (see \cite{ChCoTi} Theorem 1.15) 
proved that the singular
set of the limit space $(X,d_X)$
is empty, if all manifolds occuring in the sequence above
satisfy (a),(b) and $\fint_{B_1(x)}|\Riem|^{3/2} \leq D$ 
for all balls of radius one.
The condition
$\int_M |\Riem|^{3/2} \leq D$ prohibits non-flat cones over spheres occuring in $(X,d)$.
Theorem \ref{chcoconj} allows the occurence of such cones.
\end{rema}

The method we use to prove this theorem is as follows.
Let $(M_i,g_i(0))$ be a sequence of manifolds satisfying (a),(b) and (c).
We flow each of the $(M_i,g_i(0))$ by Ricci flow to obtain solutions
$(M_i,g_i(t))_{t \in [0,T_i)}$. Then we prove
uniform estimates (independent of $i$) for the solutions.
Once we have these estimates, we are able to take a limit of these
solutions, to obtain a new solution $(M,g(t))_{t\in (0,T)}$ where $M$ is some 
manifold. This solution will also (by construction: it is a smooth limit)
satisfy similar estimates to those obtained
for $(M_i,g_i(t))_{t \in [0,T_i)}$.
Using these estimates, we show that $(M,d(g(t))) \to (X,d_X)$
in the Gromov-Hausdorff sense as $t \searrow 0$, and that in fact
$X$ is diffeomorphic to $M$.
The most important step in this procedure is proving uniform estimates for the 
solutions $(M_i,g_i(t))_{t \in [0,T_i)}$.
The case that the $(M,g_i(0))$ satisfy (a),(b) and ($\ti c$) is reduced to the case
that the $(M,g_i(0))$ satisfy (a),(b) and (c) by a 
conformal deformation of  the starting metrics 
(which leave the starting metrics unchanged on larger and larger balls
as $i \to \infty$:
see chapter 7 for details).

The estimates we require to carry out this procedure are obtained in the following theorem
(see Theorem \ref{maintheo2}).

\begin{theo}\label{intromaintheo2}
Let $k \in \R$, $0<v_0 \in \R$, $m \in \N$ and
$(M, g_0)$ be a three (two) manifold satisfying
(a),(b) and ($\ti {\rm c}$) with constants $k$, $v_0$ and $m$ respectively.
Then there exists a $T = T(v_0,k,m) >0$ and $K = K(v_0,k,m)>0$ and
a solution $(M,g(t))_{t \in [0,T)}$ to Ricci-flow satisfying
\begin{itemize}
\item[(${\rm a}_t$)] 
$\Ricci(g(t)) \geq - K^2 \label{est2}  \ \ \forall \ \  t \in (0,T)$
\item[(${\rm b}_t$)]
$ \vol(B_1(x,t)) \geq \frac{v_0}{2} > 0 \ \ \forall  x \in M, \   \forall \ \  t \in (0,T).$
\item[(${\rm c}_t$)]
$\sup_{M} |\Riem(g(t))| \leq {K^2 \on t} \ \ \forall \ \  t \in (0,T)$ 
\item [(${\rm d}_t$)]
$e^{K^2(t-s)} d(p,q,s) \geq  d(p,q,t) 
\geq  d(p,q,s) -K^2(\sqrt t  - \sqrt s)$\hfill\break
for all $0 < s \leq t \in (0,T)$
\end{itemize}
(note that these estimates are trivial for $t=0$).

\end{theo}

\begin{rema}
A similar result was proved in the paper \cite{Si4} (see Theorem 7.1 there), under the extra assumptions
that $(M, g_0)$ has $\diam(M,g_0) \leq d_0 < \infty$ and $\Ricci(g_0) \geq
-\ep_0(d_0,v_0)$ where $\ep(d_0,v_0)>0$ is a small constant depending
on $d_0$ and $v_0$.
\end{rema}

To help us prove Theorem \ref{intromaintheo2} we prove estimates  on the rate 
at which the infimum of the Ricci curvature can decrease, and on the rate at which 
the distance function and volume of such a solution can change 
(see Lemma \ref{diamlemm} and \ref{volcorr}). 
As an application of  Theorem \ref{intromaintheo2} and these estimates
we get (Theorem \ref{realmaintheo2} in this paper).

\begin{theo}\label{introthm}
Let $k,v_0,m \in \R$ be fixed.
Let $(M_i, \upi g_0)$ be a sequence of three 
(or two) manifolds satisfying (a),(b),(c) or (a),(b),($\ti c$), (with constants $k,v_0,m$ indepedent of $i$) and let
$(X,d,x) = \lim_{i \to \infty} (M_i, d(\upi g_0), x_i)$ 
be a pointed Gromov-Hausdorff
limit of this sequence.
Let $(M_i,\upig(t))_{t \in [0,T)}$ be the solutions to Ricci-flow 
coming from the theorem above. Then (after taking a sub-sequence if necessary)
there exists a Hamilton limit solution $(M,g(t),x)_{t \in (0,T)} := \lim_{i \to \infty}
(M_i,\upig(t),x_i)_{t \in (0,T)}$ satisfying  \at, \bt, \ct, \dt, and
\begin{itemize}
\item[(i)] $(M,d(g(t)),x) \to (X,d,x)$ in the Gromov-Hausdorff sense
as $t \to 0$.
\item[(ii)] $M$ is diffeomorphic to $X$. In particular, $X$ is a manifold.
\end{itemize}
\end{theo}

As a corollary to this result and Lemma \ref{seclemma2}
we obtain the following corollary (Corollary \ref{semicoro} in this paper).

\begin{coro}
Let $(M_i, {\upi g}_0)$,$i \in \N$  be a sequence of three (or two) manifolds
satisfying (b), (c) or (b),${\rm (\ti c)}$, and
$$\Ricci(M_i,{\upi g}_0) \geq - \frac 1 i .$$
Let $(X,d_X) = \GHlim_{i \to \infty}(M_i,d(  {\upi g}_0) )$
 (notation $\GHlim$ refers to the Gromov Hausdorff limit).
Then the solution 
$(M,g(t),x)_{t \in (0,T)}$ obtained in Theorem \ref{introthm}
satisfies $$ \Ricci(g(t)) \geq 0$$ for all $t \in (0,T)$ and $(X,d_X)$ is
diffeomorphic to $(M,g(t))$  for all $t \in (0,T)$.
In particular, combining this with the results of W.X.Shi \cite{Sh2} and
R.Hamilton\cite{Ha82}, we get that
$(X,d_X)$ is diffeomorphic to  $\R^3$, $\Sphere^2 \times \R$ or $\Sphere^3$ 
modulo a  group of fixed point free isometries in the
standard metric.
\end{coro}

\vskip  0.07 true in
\section{Previous results}
\vskip  0.1 true in
We present here some previous results related to Ricci flow 
of non-smooth metrics.

In the paper \cite{Si}, the Ricci flow of continuous metrics is considered.
Estimates similar to those in Theorem \ref{intromaintheo2}  are proved.

In the paper \cite{ChTiZh}
K\"ahler Ricci flow of $L^{\infty}$ K\"ahler metrics is considered,
 estimates similar to
$(a)_t$ of Theorem \ref{intromaintheo2} are proved.

In the paper \cite{YaI} the Author considers the Ricci flow
of initial metrics which satisfy ( uniformly) an  $L^{n/2}$ 
bounds on the curvature, an $L^p$ bound on the Ricci curvature ($p> (n/2)$)
and a volume and diameter bound. He proves using Moser iteration, 
that estimates similar to
$(a)_t$ of Theorem \ref{intromaintheo2} hold under the Ricci
flow of such a metric. 

In the paper \cite{Ye}, the class of metrics
with $|\Ricci| \leq 1$ and conjugate radius bigger than $r_0$ is considered.
The authors prove estimates similar to $(a)_t$ of Theorem \ref{intromaintheo2}
once again using moser iteration.

In the paper \cite{Pe1}, the author proves an estimate of the form 
$(a)_t$ of Theorem \ref{intromaintheo2}, under
the assumtion that all neighbourhoods are alomost euclidean,
and the scalar curvature is bounded from below. Here, a blow up
argument is used, and an analysis of a backward evolving heat-type flow
(see also \cite{PeLu} and \cite{CTY} ).

In the paper  \cite{GuXu},
the author extends the results of Yang to the case that 
the manfiold is non-compact, and $\Ricci \geq -1$ and  an $L^p$ bound on the
curvature holds ($p > (n/2)$)
(see also \cite{MaLi}).

The case that the $L^{(n/2)}$ curvature is small locally, and
and an $L^p$ bound on the norm of the Ricci curvature exists, is considered
in the paper \cite{WaYu}.

The Ricci flow of compact manifolds with $\vol \geq 1$, 
$\diam \leq d_0$ and  $\Ricci \geq -\ep(d_0,n)$ 
$\ep(d_0,n) $ small is investigated  in \cite{Si4}.

\section{Methods and structure of this paper}

As explained in the introduction, we shall cheifly be concerned with
Riemannian manifolds $(M,g)$ which are contained
in $\curlT(3,k,m,v_0)$ or  $\curlTi(3,k,v_0)$, where these two spaces
are defined as follows.\begin{defi}\label{curlTdef} 
We say $(M,g) \in  \curlTi(n,k,v_0)$ if
$(M^n,g)$ is a smooth $n$-dimensional Riemannian manifold satisfying:
\begin{itemize}
\item[(a)]  $\Ricci(g) \geq k$
\item[(b)] $ \vol({\upg B}_1(x)) \geq v_0 > 0$ for all $x \in M$
\item[(c)] $\sup_M |\Riem(g)| < \infty$.
\end{itemize}
We say $(M,g) \in  \curlT(n,k,m,v_0)$ if
(a) and (b) are satisfied and the condtion (c) is replaced by
\begin{itemize}
\item[($\ti {\rm c}$)] 
Let $f:M \to \R$ be the exponential
function composed with its self $m$-times, and $\rho:M \to \R_0^+$
the distance function from a fixed base point $b$, $\rho(x) := \dist(x,b)$.
We assume that 
\begin{itemize}
\item[ (${\ti{\rm{c_1}}}$ )]
$\rho:M-B_R(b) \to \R$ is smooth for some $R>0$,
and
$\rho $ is $k$-concave there, that is
$$ \grad^2 \rho \leq k,$$ on $M-B_R(b)$ and
\item[ (${\ti{\rm{c_2}}}$)]
$$\lim_{r \to \infty} (\sup_{x \in B_r(b)}|\Riem(x)|/f(r))  = 0.$$
\end{itemize}

\end{itemize}
\end{defi}

Let us define 
$\bar \curlT(n,k,m,v_0)$ ( $\bar  \curlTi(n,k,v_0)$ )
as the set of metric spaces $(X,d_X)$ which arise as
the Gromov-Hausdorff limit of sequences whose elements are
contained in $\curlT(n,k,m,v_0)$  ( $\curlTi(n,k,v_0)$ ).
Elements of $\bar \curlT(n,k,m,v_0)$ ( $\bar  \curlTi(n,k,v_0)$ ) can be very irregular, 
and are not a priori manifolds (as we saw in the two examples of the 
introduction).
Nevertheless, they will be length spaces and do carry
some structure.
In the first part of the paper we concern ourselves only with $\bar  \curlTi(3,k,v_0)$.
Assume $(X,d_X) \in \bar \curlTi(3,k,v_0)$ is given by $(X,d_X) = \GHlim_{i \to \infty} (M^3_i,d(g_i))$
for $(M_i,g_i) \in  \curlTi(3,k,v_0)$.
In order to define a Ricci flow of $(X,d_X)$ 
we will flow each of the $(M^3_i,g_i)$ and then take a Hamilton limit of the
solutions (see \cite{Ha95a}).
The two main obstacles to this procedure are:
\begin{itemize}
\item It is possible that the solutions $(M_i,g_i(t))$  are defined only for $t \in [0,T_i)$ where $T_i \to 0$ as $i \to
\infty$.
\item In order to take this limit, we require that each of the 
solutions satisfy uniform bounds of the form
$$\sup_{M_i} |\Riem(g_i(t))| \leq |c(t)|, \ \forall \ t \in (0,T) ,$$
for some well defined common time interval 
$(0,T)$  and some function $c:(0,T) \to \R$ where
$\sup_{[R,S]} |c| < \infty$ for all $[R,S] \subset (0,T)$ 
($c(t) \to \infty$ as $t \to 0$ is allowed). Furthermore they should all
satisfy a uniform lower bound on the injectivity radius of the form
$$\inj(M,g_i(t_0)) \geq \si_0 > 0 $$ for some $t_0 \in (0,T).$
\end{itemize}
As a first step to solving these two problems,  
in Lemma \ref{globaltimelemma} of Section \ref{chapbound} we see that a (three dimensional) smooth 
solution to the Ricci flow $(M,g(t))_{t \in [0,T)}$ such that
$(M,g(t)) \in \curlTi(3,k,v_0)$ for all $t \in [0,T)$ and $\sup_{M \times[0,S]} |Riem|< \infty$ for all $S <T$
cannot become singular at time $T$.
Furthermore, a bound of the form
\begin{eqnarray}
| \Riem(g(t))|  \leq {c_0(k,v_0) \on t} \forall t \in [0,T) \cap [0,1] \label{nicey}
\end{eqnarray}
for such solutions is proved: that is, the curvature of such solutions is quickly smoothed out.

In Section 5 we prove an a priori estimate on the rate  (Lemma \ref{seclemma2}) at which the infimum
of the Ricci curvature of a solution to the Ricci flow with bounded curvature can decrease.
Note: this lemma is a non-compact version of Lemma 5.1 from \cite{Si4}.

\begin{lemm}
Let $g_0$ be a smooth metric on a 3-dimensional non-compact 
manifold $M^3$ satisfying 
\begin{eqnarray*}
&& \Ricci(g_0) \geq - { \ep_0 \on 4} g_0, \cr
&& (\sec(g_0)  \geq - { \ep_0 \on 4} g_0 )
\end{eqnarray*}
for some $0 < \ep_0 < {1 \on 100}$,  and 
let $(M,g(\cdot,t))_{t \in [0,T)}$ be a smooth solution to Ricci flow with
bounded curvature at all times.
Then  
\begin{eqnarray}
\Ricci(g(t)) & \geq -\ep_0(1 + kt) g(t) - \ep_0(1 + kt) t \Sc(g(t)) g(t), \fall t \in [0,T) \cap [0,T'), \cr
 (\sec(g(t)) & \geq  -\ep_0({1 \on 2} + kt) g(t) - \ep_0 ({1 \on 2} + kt) t \Sc(g(t)) g(t), \fall t \in [0,T)\cap [0,T'), ) 
\cr \label{nicey2}
\end{eqnarray}
where $k = 100$ and $T'=T'(100) > 0$ is a universal constant.
\end{lemm}

One of the major applications of this Lemma is:
any solution $(M,g(\cdot,t))_{t \in [0,T)}$ in $\curlTi(3,k,v_0)$ which has bounded curvature at all times
and satisfies $\Ricci(g_0) \geq -\ep_0$ 
at time zero, must also satisfy $\Sc(g(t)) \leq {c_0 \on t}$ (from \eqref{nicey}) and hence
from \eqref{nicey2}
$$ \Ricci(g(t)) \geq - 2 c_0\ep_0 \forall t \in (0,T') \cap (0,T) \cap (0,1) $$

In Section 6, we consider smooth solutions to the Ricci flow which satisfy
\begin{eqnarray}
&&\Ricci(g(t)) \geq - c_0,  \label{intro1} \cr
&& |\Riem(g(t))| t \leq c_0. \label{intro2}  \cr
\end{eqnarray}
In Lemma \ref{diamlemm}, well known bounds on the evolving distance for a solution to the Ricci flow are proved
for such solutions.

We combine this Lemma with some results on Gromov-Hausdorff convergence and a theorem of Cheeger-Colding
(from the paper \cite{ChCo}) to show (Corollary \ref{volcorr})
that such solutions can only lose volume at a controlled rate.

The results of the previous sections are then used to prove a theorem ( section 6 ) which
tells us how a priori the Ricci flow of an element $(M,g_0) \in \curlTi(3,k,v_0)$ behaves:  
see Theorem \ref{maintheo}.

In Section 8 we show that any $(M,g) \in \curlT(n,k,m,v_0)$
can be approximated in the GH sense by manifolds $(M,g_i) \in \curlTi(n,k,\ti v_0)$, $i \in \N$. 
More precisely, we show that there exists $\ti v_0 = \ti v_0(n,k,m,v_0) >0$
and $(M_i,g_i) \in  \curlTi(n,k,\ti v_0)$ with
$$(g_i)|_{B_i(x_0)}  = g|_{B_i(x_0)}$$ such that $(M,d(g_i)) \to (M,d(g))$ in the Gromov-Hausdorff sense as $i \to \infty$.
This section is independent of the rest of the paper, and requires no knowledge of the Ricci flow.

Finally, using the results of the previous two sections, we show that
a solution to the Ricci flow of $(X,d_X)$ exists, 
where  $(X,d_X)$ is the Gromov-Hausdorff limit
as $i \to \infty$ of $(M_i, d(g_i))$ where the 
$(M_i, g_i)$ are in $ \curlT(3,k,m,v_0)$, and that 
this solution satisfies certain a priori estimates.
See Theorem \ref{realmaintheo2}.

Appendix A contains some Hessian comparison principles and the proofs
thereof.
Appendix B contains a result on the rate at which distance changes under 
Ricci flow if the solution satisfies $|\Riem| \leq c/t$.

\section{ Bounding the blow up time from below using bounds on the geometry.}\label{chapbound}
\setcounter{equation}{0}
\setcounter{defi}{0}
 \numberwithin{defi}{section} 
\numberwithin{equation}{section}

An important property of the Ricci flow is that: \hfill\break
if certain geometrical quantities are controlled (bounded) on a half open finite time interval $[0,T)$, then
the solution does not become singular as $t \upto T$ and may be extended to a solution defined on the time interval
$[0,T + \ep)$ for some $\ep >0.$
As in the paper \cite{Si4}, we are interested in the question:

\begin{prob}
What elements of the geometry need to be controlled, in order to guarantee that a solution does not become singular?
\end{prob}

In \cite{Sh}, it was shown that
for $(M,g_0)$ a smooth non-compact Riemannian manifold with
$\sup_M |\Riem(g_0)| < \infty$, the Ricci flow equation
\begin{eqnarray}
&&\partt g = - 2 \Ricci(g), \label{Ricci}\cr
&&g(\cdot,0) = g_0,\nonum
\end{eqnarray}
has a short time solution
$(M,g(t))_{t \in [0,T)}$
for some $T = T(k_0,n)$ satisfying
\begin{eqnarray}
\sup_M |\Riem(g(t))| < \infty \ \forall t \in [0,T)\nonum
\end{eqnarray}
(the compact case was proved by Hamilton
in \cite{Ha82}). 
Using Shi's solution (Theorem 1.1 of \cite{Sh}), we can find a solution
$(M,g(t))_{t \in [0,T)}$
satisfying
\begin{equation}\label{max1}
\begin{cases}
\, &\sup_M |\Riem(g(t))| < \infty \ \forall \ t \in [0,T), \\
\, & \lim_{t \to T} \sup_M |\Riem(g(t))| =  \infty 
\end{cases}
\end{equation}
or
\begin{equation}\label{max2}
\begin{cases}
\, &T =  \infty,   \\
\,   &  \sup_M |\Riem(g(t))| < \infty \ \forall \ t \in [0,\infty). 
\end{cases}
\end{equation}

\begin{defi}\label{maxibc}
A solution $(M,g(t))_{t \in [0,T)}$ to Ricci flow which satisfies either \eqref{max1}
or \eqref{max2} is called a {\it maximal} solution with bounded curvature (or 
{\it maximal with BC}). 
\end{defi}
It was also shown in Shi \cite{Sh} that if $(M,g(t))_{t \in [0,T)}$ is a smooth solution with $T< \infty$ and
$\sup_{M\timess [0,T)} |\Riem| < \infty$, then 
there exists an $\ep >0$ and a solution $(M,h(t))_{t \in [0,T + \ep)}$, with $h|_{[0,T)} = g|_{[0,T)}$

So we see that a bound on the supremum of the Riemannian curvature on $M \times [0,T)$  
(that is, {\it control}
of this geometrical quantity) 
guarantees that this solution does not become singular as $t \upto T$, and that it may be
extended past time $T$ (where we are assuming here that $T< \infty$).
In the following lemma, we present 
other bounds on geometrical quantities which guarantee that a solution to the Ricci flow does
not become singular as $t \upto T$ (once again, $T< \infty$ is being assumed here).

\begin{lemm}\label{globaltimelemma}

Let $(M^{3(n)},g(t))_{t \in [0,T)}$, $T \leq 1$ be an arbitrary smooth complete
solution to Ricci flow   satisfying the conditions  
\begin{itemize}
\item[(i)]
$\Ricci(g(t))  \geq -k^2 \ \ (\curlR(g(t)) \geq -k^2),$
\item[(ii)] 
${\vol(B_1(x,t))} \geq  v_0 > 0$ for all $x \in M$,
\item[(iii)]
$\sup_M |\Riem(g(t))| < \infty$,  
\end{itemize}
 for all $t \in [0,T)$ ( notation: $\curlR$ refers to the curvature operator). 
Then there exists a $c_0 = c_0(v_0,k)$ ($c_0 =c_0(v_0,k,n)$)
such that
\begin{eqnarray*}
\sup_M |\Riem(g(t))| t \leq c_0,
\end{eqnarray*}
for all $t \in [0,T).$
In particular,  $(M^{3(n)},g(t))_{t \in [0,T)}$ is not maximal with BC.
\end{lemm}

\begin{proof}

Assume to the contrary that there exist solutions 
$(M_i, {\upig(t)})_{t \in [0,T_i)}$, $T_i \leq 1$ to Ricci flow
satisfying the conditions  (i),(ii),(iii) and such that
 \begin{eqnarray}
 \sup_{(x,t) \in M_i\timess (0,T_i)} |\Riem(\upi g)|(x,t)t  \itoinfty \infty,  \nonum
\end{eqnarray} 
or there exists some $j \in \N$ with
\begin{eqnarray}
 \sup_{(x,t) \in M_j\timess (0,T_j)} |{\Riem(\up{j}g) }|(x,t)t  =  \infty.  \nonum
\end{eqnarray} 

It is then possible to choose points $(p_i,t_i) \in M_i\timess [0,T_i)$ 
(or in  $M_j \timess [0,T_j)$: in this case we
redefine $M_i =M_j$ and $T_i = T_j$ for all $i \in \N$ and hence we do not need to treat this
case separately ) such that
\begin{eqnarray}
|\Riem(\upi g)|(p_i,t_i)t_i  
= -\ep_i + \sup_{(x,t) \in M_i\timess (0,t_i]} |\Riem(\upi g)|(x,t)t  \to \infty \nonum
\end{eqnarray}
as $i \to \infty$ where $\ep_i \to 0$ as $i \to \infty$.
Define
\begin{eqnarray}
  { \upghi}(\cdot,\hat t) := c_i {{\upig}}(\cdot,t_i +{\hat t \on c_i}),\nonum
\end{eqnarray}
 where
$c_i :=  |\Riem(\upi g)|(p_i,t_i).$
This solution to Ricci flow is defined for $ 0 \leq t_i +{\hat t \on c_i}< T_i,$
that is, at least for $0 \geq \hat t > -t_i c_i=: A_i.$
Then the solution  ${ \upghi}(\hat t)$ is defined at least for
$  \hat t \in (-A_i, 0 ).$
By the choice of $(p_i,t_i)$ we see that the solution is defined
for $\hat t > -A_i =  -t_i c_i =  -t_i|\Riem(\upi g)|(p_i,t_i) \itoinfty - \infty.$
Since $t_i \leq T_i  \leq 1,$ we also have 
\begin{eqnarray} 
c_i \itoinfty \infty , \label{cieqn}
\end{eqnarray}
in view of the fact that
 $$ t_i c_i= t_i |\Riem(\upi g)|(p_i,t_i) \itoinfty \infty .$$
Fix a constant $A \in (-A_i,0]$. For any $\hat t$ with 
$  - A_i  <  A < \hat t \leq 0$
define $s(\hat t,i) := t_i +{\hat t \on c_i}.$
Then for all such $\hat t$ we have
\begin{eqnarray}
|\Riem(\upghi)|(\cdot, \hat t) &=& {1 \on c_i} |\Riem(\upgi)|( \cdot,s(\hat t ,i)) \cr
&=&   \frac{|\Riem(\upi g)|( \cdot,s(\hat t ,i))}{|\Riem(\upi g)|(p_i,t_i)}      \cr
&=&     \frac{s |\Riem(\upi g)|( \cdot,s)}{t_i|\Riem(\upi g)|(p_i,t_i)}\frac{t_i}{s}    \cr
& \leq &  (1 + \ep_i) \frac {t_i}{s}  \cr
&=&  (1 + \ep_i){t_i \on  t_i + {\hat t \on c_i} } \cr
& \leq & (1 + \ep_i){t_i \on  t_i + {A \on c_i} } \itoinfty 1 \nonum  \\ \label{fidely}
\end{eqnarray}
in view of the definition of $(p_i,t_i)$, and $0 \leq s \leq t_i$
(follows from the definition of $s$ and the fact that $\hat t \leq 0$), and
\eqref{cieqn}.
Since  $\vol(B_1(p),\upig(t))\geq v_0> 0$ 
and $\Ricci \geq  -k^2  $ ($\curlR \geq -n^2k^2$) 
(in the case $n=3$ this is true by assumption, in the
general case it is true as all sectional curvatures are not less then $-k^2$)
we have
\begin{eqnarray}
\infty > l(n,v_0) \geq {\vol(B_r(p),\upig(t)) \on r^n} \geq \tvz(n,v_0)>0 \ \forall \ 1 >r>0 \nonum
\end{eqnarray} 
(in view of the Bishop Gromov comparison principle)
which implies the same result (for radii scaled appropriately)
for the rescaling of the manifolds: 
\begin{eqnarray}
l(n,v_0) \geq {\vol(B_r(p), \upghi(t)) \on r^n}\geq  \tvz(n,v_0) \ \forall \ \sqrt{c_i}>r>0 \label{volli}
\end{eqnarray} 

Now using 
\begin{eqnarray}
l \geq {\vol(B_r(p),\upghi(t) ) \on r^n} \geq  \tvz, \forall 0< r <1, \label{volinit}
\end{eqnarray}
we obtain a bound on the injectivity radius from below, in view
of the theorem of Cheeger/Gromov/Taylor, \cite{CGT}
$\Big($the theorem of Cheeger/Gromov/Taylor says that for a complete Riemannian manifold  $(M,g)$ with
$|\Riem| \leq 1$, we have
$$ \inj(x,g) \geq r { \vol(g, B_r(x)) \on   \vol(g, B_r(x)) + \omega_n\exp^{n-1} },$$
for all $r \leq {\pi \on 4}$:
in particular, using that $\diam(M,g) \geq N $ ,$N$ as large as we like, and $|\Riem| \leq 2$ 
for the Riemannian manifolds in question, we obtain
\begin{eqnarray}
 \inj(x,g) \geq \tvz {s^{n+1} \on l s^n +  \omega_n\exp^{n-1}} \geq
  c^2(\tvz,n) > 0 \nonum 
\end{eqnarray}
for $s = \min( (\omega_n \exp^{n-1})^{1 \on n}, {\pi \on 4} )$ $\Big)$.

This allows us to take a pointed {\it Hamilton limit} (see \cite{Ha95a}), 
which leads to a Ricci flow solution
$(\Omega,o, g(t)_{t \in (- \infty,\omega) }),$
with $|\Riem(g(t))| \leq |\Riem(o,0)| = 1,$ and $\Ricci \geq 0$ ($\curlR \geq 0$),  $\omega \geq 0$

In fact the limit solution satisfies $\curlR \geq 0$ for $n=3$ also, see Corollary 9.8 in \cite{ChKn}.

The volume ratio estimates 
\begin{eqnarray}
l \geq {\vol(B_r(p)) \on r^3} \geq \tvz \forall r >0, \label{vol}
\end{eqnarray}
 are also valid for $(\Omega,g),$
in view of \eqref{volli}.

We now apply Proposition 11.4 of \cite{Pe1} to obtain a contradiction.

\end{proof}

\section{Bounds on the Ricci curvature from below under Ricci flow in three dimensions}
\setcounter{equation}{0}                                        
The results of this section are only valid in dimensions two and three.

We prove a quantitative estimate that tells us how quickly the Ricci curvature can decrease, if we assume
at time zero that the Ricci curvature is not less than $-1$ and that the
supremum of the curvature of the evolving metric is less than infinity.
This involves modifying the argument from \cite{Si4} to the case that $M$ is non-compact. This result has similarities to the estimate of 
Hamilton-Ivey (see \cite{Ha95} or \cite{Iv} for a proof of the
Hamilton-Ivey estimate, which was
independently obtained by R.Hamilton and T.Ivey). 
%using a non-compact maximum principle for tensors.
For a general heat type equation on a non-compact manifold  $f:M \times [0,T] \to \R$
\begin{eqnarray*}
 &&\partt f  = \lap_{g(t)} f + a f + g(V ,\grad f)\cr
&&f(\cdot,0)  = f_0 \geq 0 ,
\end{eqnarray*}
it is well known that the maximum principle does not hold for general solutions
$f$, and for general $V$ and $a$. 
%Even in the case that $f$ is bounded, we do not necessarily have a maximum principle.
%For example, consider the function
%$f_0(x) = \frac 1 {1 + x^2}$ on $M = \R$.
%Let $a: \R \times \R \to \R $ be $a(x,t) = -{1 + x^2}$.
%Then there is a smooth solution $f$ to
%$$ \partt f = \lap f + a f$$ which is bounded, but 
%which satisfies $f(x,t)< 0$ for $x$ large enough and $t>0$ small.
%That is $f \geq 0$ for all $t\geq 0$ is not correct.
In the case that $a$ and $V$ are bounded, there are a number of maximum principles which
can be applied as long as the growth of $f$ is controlled, and the evolving metric
$g$ satisfies certain conditions (for example $|\partt g| \leq c $ ): see for example \cite{EcHu},
\cite{NiTa2}.
In the case of Tensors, there are also a number of Theorems which present conditions
which guarantee that the Tensor Maximum principle of Hamilton holds in a non-compact setting:
see for example Theorem 2.1 in \cite{NiTa} and Theorem 7.1 of \cite{Si}.  

In the proof of the lemma below we construct a tensor $L$ which satisfies
$\partt L \geq \lap L + N$ where $L(\cdot,0) \geq 0$ and $L(x,t) \geq \ep >0$ for all $x$ far away from
an origin, and $N(x_0,t_0)(v,v) \geq 0$ for all $v$ which satisfy $L(x_0,t_0)(v,v) = 0.$ 
This allows us to argue exactly as in the proof of the Tensor maximum principle for compact
manifolds (proved by R.Hamilton in \cite{Ha82}) to conclude that $L \geq 0$ everywhere if $L \geq 0$ at $t = 0$.

\begin{lemm}\label{seclemma2}
Let $g_0$ be a smooth metric on a 3-dimensional (or 2-dimensional) manifold $M^{3(2)}$ satisfying
$\sup_M |\Riem(g_0)| < \infty,$ 
and
 \begin{eqnarray*}
&& \Ricci(g_0) \geq - { \ep_0 \on 4} g_0   \ \  
( \sec(g_0)  \geq - { \ep_0 \on 4} g_0 )
\end{eqnarray*}
for some $0 < \ep_0 < {1 \on 100}$. 
Let $(M,g(\cdot,t))_{t \in [0,T]}$ be a solution to Ricci flow with $g(0) = g_0(\cdot)$ and
$\sup_{M \times [0,T]} | \Riem(g(t)) | < \infty.$
Then  \begin{eqnarray*}
\Ricci(g(t)) & \geq -\ep_0(1 + kt) g(t) - \ep_0(1 + kt) t \Sc(g(t)) g(t), \cr
 (\sec(g(t)) & \geq  -\ep_0({1 \on 2} + kt) g(t) - \ep_0 ({1 \on 2} + kt) t \Sc(g(t)) g(t)) ,
\end{eqnarray*}
for all  $t \in [0,T)\cap [0,T')$
where $k = 100$ and $T'> 0$ is a universal constant.
\end{lemm}

\begin{proof}
The proof is a non-compact version of the proof in \cite{Si4}.
We prove the case $n=3$
(for $n=2$ simply take $N = M \times \Sphere^1$).

Define $\ep = \ep(t) = \ep_0(1 + kt),$
and the tensor $L(t)$ by 
$$ L_{ij} := \Ricci_{ij} +   \ep \Sc t g_{ij}  + \ep g_{ij} + \si f g_{ij}$$
where $\si \leq \ep^2_0$ and $f = e^{\rho^2(1 + a t) + a t}$, $\rho(x,t): = \dist(g(t))(x_0,x)$ for some fixed $x_0$,
and $a = 1000n (1+\sup_{M \times [0,T]} | \Riem(g(t)) |)$.
We will often write $\ep$ for $\ep(t)$ (not to be confused with $\ep_0$).
Notice that 
$\ep_0 < \ep(t)\leq 2\ep_0,$ for all $t \in  [0,{1 \on k})= [0,{1 \on 100}):$ 
we will use this freely.
Then ${L_i}^j = ({R_i}^j + \ep \Sc t {\de_i}^j +  \ep{\de_i}^j + \si f{\de_i}^j)  ,$ and  
as in the paper \cite{Si4}, we calculate:
 \begin{eqnarray*}
(\partt L)_{ij} 
             &=& (\lap L)_{ij}  + N_{ij} - \si \lap f g_{ij}
   +    \si (\partt f) g_{ij}\\ 
\end{eqnarray*}
and $N_{ij}$ is (up to the constant $k = 100$)  
the same as the Tensor from the paper \cite{Si4},
$$N_{ij} := - Q_{ij} + 2R_{ik}R_{jm}g^{km} + \ep  \Sc g_{ij} +  2\ep t |\Ricci|^2 g_{ij}
+ k \ep_0 t \Sc g_{ij}                 + k \ep_0  g_{ij} - 2L_i^lR_{jl},$$
where $Q_{ij}:= 6g^{kl}R_{ik}R_{jl}   - 3\Sc R_{ij} + (\Sc^2 -2S)g_{ij}$.
For our choice of $a$ we get
 \begin{eqnarray*}
(\partt L)_{ij}  && \geq (\lap L)_{ij}  + N_{ij} 
   +    {a \on 2 }\si f g_{ij}\\ 
\end{eqnarray*}
for $ta \leq 1$ in view of the Laplacian 
comparison principle (see the Hessian comparison principle in Appendix A),
as long as $\rho^2$ is smooth in time and space where we differentiate.

In the following, we argue as in the proof of Hamilton's maximum principle,
Theorem 9.1, \cite{Ha82}.
We claim that $L_{ij}(g(t)) > 0$ for all $t \in [0,T).$
Notice that
 $f$ has exponential growth, and the other terms in the definition 
of $L$ are bounded. This guarantees that $L>0$ outside a compact set.
Hence if $L_{ij}(g(t)) > 0$ is not the case, 
then  there exist a first time and point $(p_0,t_0)$ and a direction $w_{p_0}$
for which $L(g(t))(w_{p_0},w_{p_0})(p_0,t_0) = 0.$ 

Choose coordinates about $p_0$ so that at $(p_0,t_0)$ they are orthonormal, and 
so that $\Ricci$ is diagonal at $(p_0,t_0)$ with eigenvalues
$ \la \leq \mu \leq \nu$.
Clearly $L$ is then also diagonal at $(p_0,t_0)$
with  $L_{11} =  \la +   \ep(t_0) t_0 \Sc + \ep(t_0) + \si f 
\leq L_{22} \leq L_{33},$
and so $L_{11} = 0,$ (otherwise $L(p_0,t_0) >0$ : a contradiction).
In particular, \begin{eqnarray}
N_{11}(p_0,t_0) =&&  
(\mu - \nu)^2 + \la(\mu + \nu) +  2\ep t \la^2 + 2\ep t\mu^2 + 2\ep t \nu^2 \cr
&&+
 \ep  \Sc g_{ij} + k \ep_0 t \Sc g_{ij} + k \ep_0  g_{ij} \nonum  \\\label{N}
\end{eqnarray}
in view of the definition of $Q$ 
(see \cite{Ha82} Corollary 8.2, Theorems 8.3,8.4)
and the fact that $L_{11}=0$.
As in \cite{Si4}, we will show that 
${\ti N}_{11}(p_0,t_0) =  N_{11}(p_0,t_0) + {a \on 2}\si f(p_0,t_0)
>0$ which, as we will show, leads to a contradiction.
Notice that $\Sc(\cdot, 0) \geq -\ep_0$ and 
$\sup_{M \times [0,T]} |\Riem | \leq a$ on $[0,T)$ implies
that $\Sc (\cdot,t) \geq -\ep_0$ for all $t \in [0,T)$
from the non-compact maximum principle for 
functions ( this may be seen as follows: i) $\partt (\Sc + \ep_0 + \si f) \geq \lap (\Sc + \ep_0 + \si f)$ and 
$(\Sc + \ep_0 + \si f)(x, \cdot) >0$ 
for $d(x,x_0)$ large enough, where here $f $ is as above, ii) this implies
$\Sc \geq - \ep_0 -\si f$ for all $t \in [0,T)$, iii) $\si>0$ was arbitrary). 
Then $L_{11} = 0 \Rightarrow \la =  -\ep t_0  \Sc - \ep - \si f \leq 0$ for
$t_0 \leq 1$, and hence $\mu + \nu \geq \Sc \geq -\ep$. We will use these
facts freely below.
Substituting  
$\la = -\ep t_0  \Sc - \ep - \si f$ (at $(p_0,t_0)$ ) 
into \eqref{N}, we get 
\begin{eqnarray*}
 N_{11}(p_0,t_0)  & = & (u-v)^2 + (-\ep t_0 \Sc - \ep - \si f )(\mu + \nu)
+ 2 \ep t_0 (\la^2 + \mu^2 + \nu^2) \cr
&&+
\ep  \Sc  + k \ep_0 t \Sc g_{ij} + k \ep_0 
 \cr 
&\geq&  \ep t_0 (   -( \la + \mu + \nu)(\mu + \nu) + 2 \la^2 + 2\mu
^2 + 2\nu^2) - (\ep + \si f)(\mu + \nu)  \cr
 && + \ep  \Sc  + k \ep_0 t_0 \Sc  + k \ep_0 
\cr
&=& \ep t_0 (   -( \la + \mu + \nu)(\mu + \nu) + 2 \la^2 + 2\mu^2 + 2\nu^2) \cr
&& + (-\ep^2 t_0 + k \ep_0 t_0 ) \Sc - \ep^2 - \si \ep f
+ k \ep_0 - \si f (\mu + \nu)  \cr
 &\geq&  \ep t_0 ( \la  \ep  + 2 \la^2 ) \cr
&&  - \si f \ep
+ (k-1) \ep_0 - \si f (\mu + \nu), \cr
\end{eqnarray*}
where here we have used that $ \Sc \geq - \ep$ and $ - \la(\mu + \nu) \geq \la \ep$
in the last inequality (which follows from $\mu + \nu \geq \Sc \geq -\ep$ and
$\la \leq 0$).
Hence
\begin{eqnarray}
N_{11}(p_0,t_0) + {a \on 2} \si f  & > & 0, \label{N11}
\end{eqnarray}
since $\mu + \nu \leq {a \on 100}$.

% \ep t_0 ( -\la \Sc ) + \ep t_0 \la^2   
%+\ep t_0 ( \mu^2 + \nu^2 + 2 \la^2 - 2 \mu \nu) \cr
%&& +  (-\ep^2 t_0 + k \ep_0 t_0 ) \Sc - \ep^2 - \si f \ep + k \ep_0 - \si f (\mu + \nu) \cr
%&& =  \ep t_0 (  \ep(t_0) t_0  \Sc(x_0,t_0) + \ep(t_0) +\si f(x_0,t_0))
%   (\Sc ) + \ep t_0 \la^2   +\ep t_0 ( \mu^2 + \nu^2 + 2 \la^2 
%- 2 \mu \nu)+ k \ep_0 - \si f (\mu + \nu) \cr
%&& +  (-\ep^2 t_0 + k \ep_0 t_0 ) \Sc - \ep^2 - \si f \ep \cr
%&&\geq  - \si f \ep + (k-1) \ep_0 - \si f (\mu + \nu)
%\end{eqnarray*}
%where here we have used a number of times that  
%$$\la(x_0,t_0) = -\ep(t_0) t_0  \Sc(x_0,t_0) - \ep(t_0) -\si f(x_0,t_0)$$

The rest of the proof is standard (see  \cite{Ha82} Theorem 9.1):
extend $ w(p_0,t_0) = \parti{}{x^1}(p_0,t_0)$ in space to a vector field $w(\cdot)$ in a small
neighbourhood of $p_0$ so that $^{g(t_0)}\grad w(\cdot)(p_0,t_0) = 0,$ and let
$w(\cdot,t) = w(\cdot).$
Then
$$0 \geq  (\partt L(w,w))(p_0,t_0) \geq  (\lap L(w,w))(p_0,t_0) + N(w,w) > 0 ,$$
which is a contradiction. 

If $\rho^2$ is not differentiable at $(p_0,t_0)$ then we may use the trick of Calabi:
 
let $\ga:[0,l = \rho(p_0,t) ] \to M$ be a Geodesic from $x_0$ to $p_0$ realising the distance:
and parametrised by distance, so that 
$\rho(\ga(s),t) = L_t(\ga|_{[0,s]})= s$, where $L_t$ is the length of a curve measured using $g(t))$.
Since $\rho$ is not differentiable at $p_0$ it must be that $p_0$ is a cut point of $x_0$.
Set $\ti \rho(x,t) := \rho(\ga(r),t) + \dist(g(t))(\ga(r), x)$ for some small fixed $r>0$.
Then in a parabolic neighbourhood of $(p_0,t_0)$, 
$\ti \rho$ is smooth. 
%(this is because: $p_0$ ist not a cut point of $y_0 := \ga(r)$: if it were, then 
%there would be a minimizing geodesic $\si:[r,l] \to M$ from  $\si(r) = y_0$ to $\si(l) = p_0$ for which 
%$\si(l)$ is a cut point of $\si(r)$ along the curve $\si.$
%In particular $\si(r)$ is a cut point of $\si(l)$ along the curve $\si^{-1}:[r,l] \to M$
%$\si^{-1}(s) = \si( -s +(l+r))$.
%If $\si:[r,l] \to M$ is equal to $\ga:[r,l] \to M$ then we obtain a contradiction to the definition of 'cut point', 
%as we could then extend $\si^{-1}$ to  a minimizing Geodesic $\si^{-1}:[r,l+r] \to M$.
%But then by pasting $\ga:[0,r] \to M$ and $\si:[r,l] \to M$ together, we get a non-smooth curve, which
%is a Length minimizing curve, which
%must therefore be a geodesic: this is also a contradiction)

Furthermore, from the triangle inequality,
$\ti \rho(x,t) \geq \rho(x,t)$. Also, $\ti \rho(p_0,t_0) = \rho(p_0,t_0)$.
Define $\ti L$ by 
 $$\ti L_{ij} := \Ricci_{ij}  + \ep \Sc t g_{ij} + \ep g_{ij} + \si \ti f g_{ij},$$
where $\ti f = e^{\ti \rho^2(1 + a t) + a t}$.
Then we have just shown that $\ti L \geq L$ and that
$\ti L(p_0,t_0) = L(p_0,t_0)$ and so we argue with $\ti L$ instead of $L$.
At $(p_0,t_0)$ we have $\partt \ti \rho \leq {a \on 50} \ti \rho$ and $\lap \ti \rho^2 \leq  {a \on 50}$
 (if we choose $r$ small enough): that is $\rho$ and $\ti \rho$
satisfy the same inequalities at $(p_0,t_0)$ (up to the constant $50$).

Hence we may argue as above to obtain a contradiction.

Now letting $\si$ go to zero, we get
$\Ricci_{ij} +   \ep \Sc t g_{ij}  + \ep g_{ij} \geq 0$ as long as $t a \leq 1$ and $t k \leq 1$.
But then, we may argue as above starting at $t_0 = {1 \on a}$, but now with $f_1$ in place of $f$,
$f_1 = e^{\rho^2(1+a(t-t_0)) + a(t-t_0) }$ to obtain the same result on
$[0,2t_0]$ as long as $tk \leq 1$. Continuing in this way, we see that
 $\Ricci_{ij} +   \ep \Sc t g_{ij}  + \ep g_{ij} \geq 0$ as long as  $tk \leq 1$.

The case for the sectional curvatures is similar:
from \cite{Ha86}, Sec. 5, we know that the reaction equations
for the curvature operator are
\begin{eqnarray*}
\partt \al &=& \al^2 + \be\ga, \\
\partt \be &=& \be^2 + \al \ga, \\
\partt \ga &=& \ga^2 + \al \be.
\end{eqnarray*}

It is shown in \cite{Si4} (in the proof of the compact version of this Lemma) that (for $\ep(t) := {1 \on  2}(\ep_0 + kt)$)
either 
$\partt (\al + \ep t \Sc + \ep )  >0$ or
\begin{eqnarray}
\partt (\al + \ep t \Sc + \ep )   
\geq  \ep ( \al + \be) + k \ep_0 t \Sc + k \ep_0  +   (\al + \ep t \Sc + \ep) \ga.  \label{alphy}    
\end{eqnarray}
Also $f := e^{\rho^2 (1 + at) + at}$ satisfies
\begin{eqnarray*}
\partt f \geq \lap f + {a \on 2}  f \cr
\end{eqnarray*}
at the points where $f$ is smooth and $ta \leq 1$.
So the ordinary differential equation for $f$ satisfies
\begin{eqnarray}
\partt f \geq {a \on 2}  f \label{betty}
\end{eqnarray}
at the points where $f$ is smooth and $ta \leq 1$.

Since $f$ is exponential in distance, the points
where $ \al + \ep t \Sc + \ep + \si f \leq 0$ is a compact set.
Hence, if $\al + \ep t \Sc + \ep + \si f>0$ is not true, then there must exist a first time and point $(p_0,t)$ where
this fails. 
At such a point $(p_0,t)$  we have (from \eqref{alphy} and \eqref{betty}):
\begin{eqnarray}
\partt (\al + \ep t \Sc + \ep + \si f )  &  \geq &     \ep ( \al + \be) + k \ep_0 t \Sc + k \ep_0  +   (\al + \ep t \Sc + \ep) \ga + {a \on 2} \si  f \cr 
& = &  \ep ( \al + \be) + k \ep_0 t \Sc + k \ep_0  - \si f \ga  + {a \on 2} \si  f \cr
& \geq & 2\ep \al + k \ep_0 t \Sc + k \ep_0  + {a \on 4} \si  f \nonum \\
\label{inserthere}
\end{eqnarray}
as long as $ta \leq 1$, where we have used that  $\al  + \ep t \Sc  + \ep =  - \si f$,
and that $|\ga| \leq {a \on 100}$.
Using  $\al  + \ep t \Sc  + \ep = -\si f$ again, we get  
\begin{eqnarray*}
 2\ep \al + k \ep_0 t \Sc + k \ep_0   + {a \on 4} \si  f  
&= & 2\ep( - \ep t \Sc - \ep - \si f)   + k \ep_0 t \Sc + k \ep_0  + {a \on 4} \si  f \cr
& \geq& (k-2) \ep_0 + {a \on 4} \si  f    \cr
&>&  0,
\end{eqnarray*}
 since $\Sc \geq -3 \ep_0$ is preserved by the flow,
and $t \leq {1 \on k}$.
Hence, inserting this into \eqref{inserthere}
we get
\begin{eqnarray*}
\partt (\al + \ep t \Sc + \ep +\si f)  &  > & 0
\end{eqnarray*}
at a point where $\al  + \ep t \Sc  + \ep + \si f = 0.$
Choose an orthonormal basis for the two forms at $(p_0,t_0)$:
$\phi^1 = (\phi^1)_{ij} dx^i \wedge dx^j,\phi^2 = (\phi^2)_{ij} dx^i \wedge dx^j,
\phi^3 = (\phi^3)_{ij} dx^i \wedge dx^j$ (time independent by definition)
for which the curvature operator is diagonal, and assume that $\curlR(\phi^1, \phi^1) = 
R^{ijkl} {\phi^1}_{ij} {\phi^1}_{ij} $ is the smallest
Eigen-value of the curvature operator $\curlR$.
Then we have
\begin{eqnarray*}
&&\partt( R^{ijkl}(p,t) (\phi^1)_{ij} (\phi^1)_{kl} +  \ep t \Sc + \ep + \si f )_{(p_0,t_0)} \cr
&& \ \ \ > (\lap R)^{ijkl}(p_0,t_0) (\phi^1)_{ij} (\phi^1)_{kl} 
+ \lap(  \ep t \Sc + \ep + \si f)_{(p_0,t_0)}. \cr
\end{eqnarray*}

Using the maximum principle, we obtain the result by arguing as in the 
case of the Ricci curvature above
(once again, if this inequality is violated at some point and first time, then we may need to 
modify $\rho$ in order to make sure that it is smooth,
as in the argument above for the Ricci curvature).
\end{proof}

\section{Bounding the distance and volume growth in  terms of the curvature}

\setcounter{equation}{0}

The results of this section hold for all dimensions.

\begin{lemm}\label{diamlemm}
Let $(M^n,g(t))_{t \in [0,T)}$ be a smooth solution to Ricci flow
with 
\begin{eqnarray}
&&\Ricci(g(t)) \geq - 1, \cr
&& | \Riem(g(t))| t \leq c_0.\nonum \\ \label{bobo2}
%&& \vol(B_r(x,t)) \on r^n \geq v_0  \label{bobo2}
\end{eqnarray}
Then
\begin{eqnarray} 
e^{c_1(c_0,n) (t-s)} d(p,q,s) \geq  d(p,q,t) \geq  d(p,q,s) -c_2(n,c_0) (\sqrt t  - \sqrt s)\label{diogo}
\end{eqnarray}
for all $0\leq s \leq t \in [0,T)$.
\end{lemm} 

\begin{proof}
These results essentially follow from  \cite{Ha95}, theorem 17.2 (with a slight modification of the proof suggested by the editors in \cite{CCCY} : 
see Appendix B)
and \cite{Ha95}, Lemma 17.3: see Appendix B for a proof.
\end{proof}

\begin{coro}\label{volcorr}
Let  $(M^n,{g}(t))_{t \in [0,T)}$ be an arbitrary smooth solution to Ricci flow ($g(0) = g_0$)
satisfying the conditions
\ref{bobo2}
and assume that there exists $v_0 >0$  such that 
\begin{eqnarray*}
&&\vol(B_1(x,0)) \geq v_0 >0  \ \ \forall x \  \in M.
\end{eqnarray*}
Then there exists an $S = S(c_0,v_0,n) >0$ such that  
\begin{eqnarray*}
&&{\vol(B_1(x,t))} \geq {2 v_0 \on 3} >0  
\ \ \  \forall x \in M, \ \forall t \in [0,S) \intersect [0,T).
\end{eqnarray*}
Notice that this then implies 
$$ {\vol(B_r(x,t)) \on r^n} \geq { 2 e^{-n} v_0 \on 3} \  \forall 1>r>0,$$
in view of the Bishop Gromov comparison principle.  
\end{coro}

\begin{proof}
If this were not the case, then there exist solutions 
 $(M_i^n,{^i g}(t))_{t \in [0,T_i)}$ satisfying the stated conditions and there exist
$t_i \in [0,T_i),$ $t_i \itoinfty 0$
and points $p_i \in M_i$ such that
$ {\vol(B_1(p_i,t_i))} < {2 v_0 \on 3}.$
A subsequence of $(M_i,d({^i g}(0)),p_i )$ converges to 
$(Y,d,p)$ in the pointed Gromov-Hausdorff limit.
Clearly then  $(M_i,d({^i g}(t_i)),p_i )$ also converges to 
$(Y,d,p)$, in view of the characterisation of Gromov Hausdorff convergence 
given in Corollary 7.3.28 of \cite{BuBuIv},  
and the estimates \eqref{diogo} (since $t_i \to 0$).
The Theorem of Cheeeger and Colding says that volume is continuous under the limit of non-collapsing
spaces with Ricci curvature bounded from below:
$$\lim_{i \to \infty} {\vol(B_1(p_i,t_i)} = \Haus^n(B_1(p)) = \lim_{i \to \infty} {\vol(B_1(p_i,0)}.$$
But this is a contradiction as we then 
have 
$$  { 2 v_0 \on 3}  > {\vol(B_1(p_i,t_i)} \to  \Haus^n(B_1(p))
=  \lim_{i \to \infty} {\vol(B_1(p_i,0)} > v_0.
$$
\end{proof}

\section{Non collapsed non compact three manifolds with curvature bounded from below.}

\setcounter{equation}{0}

The results of this section are only valid for dimensions two and three-

%&>&  \lap L_{ij}
%\end{eqnarray*}

\begin{theo}\label{maintheo}
Let $(M,g_0)$ be a complete smooth three (or two) manifold without boundary in $\curlTi(3,k,v_0)$: that is
\begin{eqnarray*}
\mbox{\rm (a)} && \Ricci(g_0) \geq  k, \cr
\mbox{\rm (b)} &&\vol({{}^{g_0}B}_1(x))\geq v_0  >0, \ \forall x \in M \cr 
\mbox{\rm (c)} && \sup_M |\Riem(g_0) | < \infty.
\end{eqnarray*}
Then there exists an $S = S(v_0,k) >0$ and $K = K(v_0,k)$ and a solution
$(M,g(t))_{t \in [0,T)}$ to Ricci-flow which satisfies
$T\geq S,$
and \begin{eqnarray}
\mbox{\rm (a$_t$)}&&\Ricci(g(t)) \geq -K^2 \ \forall  t \in (0,T) \cr
\mbox{\rm (b$_t$)}&&\vol ({{}^{g_t}B}_1(x))\geq {v_0 \on 2}  >0, \ \forall x \in M,
\forall  t \in (0,T)  \cr
\mbox{\rm (c$_t$)}&&\sup_{M} |\Riem(g(t))| \leq {K^2 \on t}, \forall  t \in (0,T) \cr
\mbox{\rm (d$_t$)}&& e^{c_1(c_0,n) (t-s)} d(p,q,s) \geq  d(p,q,t) \geq  d(p,q,s) -c_2(n,c_0) (\sqrt t  - \sqrt s)  \cr
&& \ \  \forall  \  0< s \leq  t \in (0,T) \nonum \\
\label{theimes}
\end{eqnarray}
\end{theo}
(note that the estimates are trivial for $t =0$).

\begin{proof}
We assume $n=3$. The argument for $n=2$ is the same.
Before proving the theorem rigorously, we present a sketch of the proof
which leaves out the technical details.
This should give the reader a clear picture of the structure
of the proof.
As a first step, we
scale the metric by a large constant, so that
 $\Ricci(g_0) \geq - \ep$ for a small $\ep = \ep(v_0,k)>0$.
The condition 
\begin{eqnarray}
\vol({\up{g_0} B}_r(x))  \geq  {\tvz}r^3 \ \forall \ 0<r \leq 1
\end{eqnarray}
for some $\tvz = \tvz(v_0,k)>0$, which is true in view of the Bishop-Gromov volume comparison principle,
remains valid under this scaling.

Now flow this metric for a maximal amount of time.
Let $[0,T_M)$ be the maximal time interval
for which the flow exists and
\begin{eqnarray}
\inf_{x \in M} \vol(B_1(x,t)) & > & {\tvz \on 2}, \label{maximal1}\\
\inf_M \Ricci(g(x,t)) &>& -1\label{maximal2}
\end{eqnarray}
for all $t \in [0,T_M)$.
Using the maximum principle and standard ODE estimates, one shows
easily that $T_M>0$.
The aim is now to show that $T_M\geq S$ for some $S = S(\ti v_0)>0$. 
From the \ref{globaltimelemma} we see that
if $T_M \geq 1$ then the estimates \at,\bt and \ct are satisfied.
So w.l.o.g $T_M \leq 1$.
From the \ref{globaltimelemma} again,
$$|\Riem(g(t)| \leq {c_0(\ti v_0) \on t}$$
for all $t \in (0,T_M)$.
Using Lemma \ref{seclemma2} we see that
$\Ricci \geq - 2\ep R t - 2\ep$ for all $t \in [0,T')\intersect (0,T_M)$
for some universal constant $T'>0$.
But these two estimates combined imply 
$\Ricci \geq -{1 \on 2}$
for all $t \in [0,T')\intersect (0,T_M)$ if  $2\ep c_0 \leq {1 \on 4}$
(we assume $c_0 >1$). We assume that we have chosen $\ep$ small enough,
in order that this estimate holds.
Similarly, using 
\ref{volcorr}, there exists a 
$T''= T''(\tvz,c_0)>0$,
such that $\vol(B_1(x,t)) > {2\tvz \on 3}$ for all 
$t \in [0,T') \cap [0,T'') \cap [0,T_M)$.
If $T_M < \min(T',T'')$, then we obtain a contradiction to the
definition of $T_M$
($T_M$ should be thought of as the first time where at least one
of the conditions \ref{maximal1} or \ref{maximal2} is violated).
Hence $T_M \geq \min(T',T'')= :S$.
But then we may use \ref{globaltimelemma} again to
show that \at,\bt,\ct are satisfied on $(0,S)$.
Scaling back to the orginal etsimates leads to rescaled estimates
 \at,\bt,\ct (with other constants).
\dt follows immediately from 
Lemma \ref{diamlemm}.
Now we prove the Theorem rigorously.

By the Bishop-Gromov volume comparison principle, we have
\begin{eqnarray}
\vol({\up{g_0} B}_r(x))  \geq  {\tvz}r^3 \ \forall \ 0<r \leq 1, \label{volli2}
\end{eqnarray}
for some $\tvz = \tvz(v_0,k)>0$.
Rescale the metric by the  constant $1000 c_0$
so that $\Ricci(g_0) \geq - \ep$ where $\ep = {1 \on 1000 c_0}$ and 
$c_0 = c_0(3,-{1 },{\tvz \on 2} )$  is from the Lemma \ref{globaltimelemma}.
Notice that
\eqref{volli2} is still true for this new rescaled metric, as we have scaled
by a constant which is larger than $1$.
We denote our rescaled metric also by $g_0$.

From the work of W.Shi (see \cite{Sh}, main Theorem) we know that there exists a solution
$(M,g(t))_{t \in [0,T)} $ to Ricci flow, with $g(0) = g_0$,
$$\sup_{M} |\Riem(g(t))| < \infty$$
for all $t \in [0,T)$. Without loss of generality, $(M,g(t))_{t \in [0,T)}$ is a maximal
solution with BC in the sense of Definition \ref{maxibc}. 
Let $T_M$ be the supremum over all $S \leq T$ such that:
\begin{eqnarray}
\inf_{x \in M} \vol(B_1(x,t)) & > & {\tvz \on 2}, \label{maximal11}\\
\inf_M \Ricci(g(x,t)) &>& -1\label{maximal12}
\end{eqnarray}
for all $t \in [0,S)\cap[0,T)$.
First we show that $T_M >0$.
We have bounded curvature on compact time intervals,
$N(\de):= \sup_{M \times [0,T-\de]} |\Riem(g(t))| < \infty,$ and hence
\begin{eqnarray}
|\partt g| \leq C(N) \label{ddtest}
\end{eqnarray}
on such time intervals, which implies 
$ \vol(B_1(x,t)) \geq \vol(B_1(x,0))(1 - \si)$ for 
$t \leq H$ $H = H(\si,N)$ small enough
($\si >0$ is an arbitrary constant).
%$\
%Big($
%we see this as follows: for arbitrary $r>0$ we can find a $c(r)>0$ such that
%for all $t\leq c(r)$
%$B_1(x,0) \subset B_{1+r}(x,t)$ due to the estimate \eqref{ddtest},
%and hence 
%$$ \vol(B_{1}(x,t)) \geq {V_{-1}(1) \on V_{-1}(1+r)}  \vol(B_{1+r}(x,t))
%\geq  (1 - r){V_{-1}(1) \on V_{-1}(1+r)} \vol(B_{1}(x,0))
%$$
%for $t\leq c(r)$. Here is   $V_k(s)$ the volume of a ball of radius
%$s$ in the standard $n$-dim. space with 
Also, 
\begin{eqnarray*}
\partt \Ricci &\geq & \up{g} \lap \Ricci -c(N)g
\end{eqnarray*}
which implies 
(choose $a = a(N,\ep)$ large enough)
\begin{eqnarray*}
\partt (\Ricci + \ep \exp^{( \rho^2(1+at) +at)}g) \geq 
\lap (\Ricci + \ep \exp^{( \rho^2(1+at) +at)}g) 
\end{eqnarray*}
for $t \leq K$, $K= K(N,\ep)$ small enough
(see argument in the proof of Lemma \ref{seclemma2}) at all points where $\rho^2$ is differentiable.
Since $\Ricci + \ep \exp^{( \rho^2(1+at) +at)}g $ must take its infimum at an 
interior point, we get (arguing as in the proof of Lemma \ref{seclemma2} for some
fixed base point $x_0$)
$\Ricci \geq -2\ep $ for 
$t \leq K$ small enough as the base point was arbitrary.
This means, that the conditions are not violated for a short time.

Due to  Lemma \ref{globaltimelemma} we have $T_M <T$: if $T_M = T$, 
then we could extend the solution to the time interval 
$[0,T + \ep)$ for some small $\ep$ using the result of 
Shi (see the discussion at the beginning of section 2) and Lemma \ref{globaltimelemma}, which would contradict the definition of $T$.
W.l.o.g. $T_M \leq 1$ : otherwise we may apply Lemma 
\ref{globaltimelemma} to immediately obtain the result.
From the same lemma (Lemma \ref{globaltimelemma}), we know that there exists a $c_0 = 
c_0({\tvz \on 2})$
such that $|\Riem|(t) \leq {c_0 \on t},$ for all $ t \in [0,T_M)$.
Using Lemma \ref{seclemma2} and the fact
that $\Sc(t) \leq {c_0 \on t}$ (combined with the choice of $c_0$)
 we see that there exists a global constant 
$T'$ such that
$\Ricci \geq -{1 \on 2}$ for all $t \in [0,T_M)\cap [0,T')$. 
So the Ricci curvature condition is not violated on this time interval.
Furthermore, in view of Corollary \ref{volcorr}, there exists a 
$T''= T''(\tvz)>0$,
such that $\vol(B_1(x,t)) > {2\tvz \on 3}$ for all $t \in [0,T'') \cap [0,T') \cap [0,T_M) $ for all $x \in M$.
So the volume condition is not violated on this time interval.

From the definition of $T_M,T'$ and $ T''$ we have 
$T\geq T_M \geq \min(T'',T')=: S(v_0)$

So we have a well defined time interval for which the conditions
\eqref{maximal11} and \eqref{maximal12} are not violated.
Furthermore, the curvature is like $c_0 \on t$ on this time interval.
Hence we have $$|\Riem(g(t))| \leq {c_0 \on t}$$ and
$$ \Ricci \geq -1$$ for all $t \leq S$.
Now we rescale the metric back, to obtain the result:
the rescaled solution
$h(\cdot,\hat t) = ({1 / 1000 c_0})g(\cdot, 1000 c_0 \hat t)$
is the desired solution. Its' initial value is given by $g_0$
($g_0$ is as in the beginning of the proof of the theorem) and it satisfies
the required estimates by scaling ( $|\Riem \leq {c_0 \on t}$ is a scale invariant inequality,
and the estimate '{\it  $ \Ricci \geq -1$ for all $t \leq S$}' scales to  '{\it $ \Ricci \geq -1000c_0$
 for all $t \leq S / (c_0 1000)$}').
That the volume of a ball of radius one is larger than $v_0 /2$ for the evolving metric  
follows from the corollary of the previous section (after shortening the time interval if necessary).
The estimate $\rm{ (d_t)}$ follows immediately from Lemma \ref{diamlemm}.

\end{proof}

\section{ Conformal deformation of non-collapsed manifolds with $\Ricci \geq -1$}
Let $(M,g)$ be a manifold satisfying (b) and (c):
\begin{itemize}
\item[(b)] $\Ricci \geq -k$,
\item[(c)] $\vol(B_1(x)) \geq v_0 >0$ for all $x \in M$. 
\end{itemize}
We wish to modify the metric $g$ to a new metric $g_i$ so that
\begin{itemize}
\item $g = g_i$ on $B_i(b)$,
\item $\Ricci \geq -{\ti k}(k,n,v_0)$,
\item $\sup_M |\Riem(g_i)| < \infty$,
\item $\vol({}^{g_i} B_1(x) ) \geq \ti v_0(k,n,v_0) >0$ for all $x \in M$,
\end{itemize}
where $b$ is a fixed origin and $i \in \N$.
In the next section we will apply the results of the previous sections in order to flow the $g_i$'s, and then we will take a limit in $i$ of the resulting solutions.

For convenience, we introduce the following notation.
\begin{defi}
Let $h: \R_0^+ \to \R^+$ be a function. We say that $h$ is a {\it function with controlled growth }
if  
$$h(x) \leq (\exp \of \exp \of \ldots \exp)(x),$$
where the function on the right hand side is the composition
of $\exp$ $m$ times, and $m$ is a fixed number in $\N$. 
We call functions of the type appearing in the right hand side an 
{\it exponential comparison function}.
\end{defi}

We require the following help Lemma about exponential comparison functions.

\begin{lemm}\label{littely}
Let $h$ be an  exponential comparison function, $0<k \in \R$, $ 0 \leq p \in \R$. 
Let $h_i:\R^+ \to \R^+$, $h_i(x):= h( (x-i)^4_+)$.
We have
\begin{eqnarray}
|r|^p + |h(r)|^p + |h'(r)|^p + |h''(r)|^p && \leq c_{k,p} e^{kh(r)}    \ \ \mbox{  for all  } r \in \R^+, \cr
|h_i(y)|&& \leq c_{k}  e^{kh_i(y)}   \ \ \mbox{  for all  }  \ y \in \R^+, \cr
|h_i'(y)| && \leq c_{k} e^{kh_i(y)}  \ \ \mbox{  for all  }  \ y \in \R^+, \cr
|h_i''(y)| && \leq  c_{k} e^{kh_i(y)}  \ \ \mbox{  for all  } \ i<y \in   \R^+, \nonum \\
\label{thehs2}
\end{eqnarray}
for some constants $c_p$, $c_{k,p}$, depending on $p$, respectively 
$k$ and $p$,  and the function $h$ {\bf but not on} $i$.
\end{lemm}
\proof 
The first estimate follows from the definition of an exponential comparison function $h$ and the fact
that $|y|^q \leq c_{q,k} e^{k |y|}$ for $q \geq 0$, for some constant $c_{q,k}$.

The next estimate follows from the definition of $h_i$ and the first estimate:
 \begin{eqnarray*}
h_i(y)&& = h( (y-i)^4_+) \cr
&& \leq  c_k e^{k h( (y-i)^4_+)} \cr
&& = c_k  e^{kh_i(y)}.
\end{eqnarray*}
The third estimate may be seen as follows

\begin{eqnarray*}
|h_i'(x)| && = |4(x-i)^3_+ h'((x-i)_+^4)| \cr
&& \leq c_k |4(x-i)_+|^3|e^{k h( (x-i)_+^2)}| \cr
&& = c_k |4(x-i)_+|^3| e^{k h_i(x)}| \cr
&& \leq c_k | e^{kh((x-i)^4_+)}|| e^{k h_i(x)}| \cr
&& = c_k | e^{k h_i(x)}| |e^{k h_i(x)}|  \cr
&& = c_k |e^{2k h_i(x)}|,
\end{eqnarray*}
where we have freely used the first estimate.
Replacing $k$ by $(k/2)$ we obtain the desired estimate.
The method for estimating $(h_i)''$ is similar:
\begin{eqnarray*}
|h_i''(x)| &&=  |(  4(x-i)_+^3 h'((x-i)^4) )'| \cr
&& = |12(x-i)^2_+ h'((x-i)^4) + 16(x-i)_+^6h''((x-i)^4)| \cr
&& \leq 12(x-i)^2 |h'((x-i)^4)| + 16 |(x-i)^6 h''((x-i)^2)| \cr
&& \leq c_k |e^{k h( (x-i)^4)}|\cr
&& = c_k |e^{k h_i(x)}|.
\end{eqnarray*}

$\Box$

\begin{defi}
Let $(M,g)$ be a Riemannian manifold and let 
$f(r):= \sup_{B_r(b)} |\Riem(g)|$ for some fixed point $b \in M$.
We say that $(M,g)$ has {\it controlled geometry at infinity} if 
\begin{itemize} 
\item $f:\R_0^+ \to \R^+$ is a function
with controlled growth
\item the distance function $\rho:(M - B_R(b)) \to \R$,$\rho(x) = \dist(g)(x,b)$is smooth for some
$R>0$ and $k-concave $ there, that is $\grad^2 \rho \leq kg$ on
$(M - B_R(b))$.
\end{itemize}
\end{defi}

\begin{theo}
Let $(M,g)$ be a smooth Riemannian manifold with controlled geometry at
infinity satisfying
\begin{itemize}
\item[(b)] $\Ricci \geq -k$
\item[(c)] $\vol(B_1(x)) \geq v_0 >0$ for all $x \in M$. 
\end{itemize}
Then there exists a family of smooth Riemannian metrics $\{g_i\}$, $i \in \N$
 on $M$ satisfying
\begin{eqnarray*}
&&g_i = g \ \mbox{ for all } x \in {B_i(p_0)}, \cr
&&\Ricci(g_i) \geq -c(n,k) g_i,  \cr
&&\vol(B_1(x_0),g_i) \geq \ti v_0  \  \mbox{ for all } x  \in M, \cr 
&&\sup_M {{}^{g_i} |}\Riem(g_i)| < \infty.
\end{eqnarray*}
\end{theo}
\begin{proof}

Let $\de <<1 $ be fixed for the rest of this section.
Let $\phi = h_i$ for some $i \in \N$, where $h_i$ is 
as in Lemma \ref{littely}. Note that 
$\phi$ is a non-decreasing function $\phi:\R^+ \to \R^+ $.
Let $b$ be a fixed base point in $M$.
$\rho: M \to \R^+$ is the distance function with respect to $b$:
$\rho(x) := \dist(x,b)$.
Let $x_0$ be an arbitrary point in $M$ and set $\rho_0: = \rho(x_0)$. 
%Without loss of generality for this argument
%$\rho_0  ($\rho_0$ should be thought of as being very
%large).
Let $\ga:[0,\rho_0] \to M$ be  a minimizing geodesic from $b$ to $x_0$ (with unit speed).
So $d(\ga(s), \ga(u)) = |u-s|$ for all $u,s \in [0,\rho_0]$.
As $M$ is complete, we may extend this smoothly to a geodesic
$\ga:[0,\infty) \to M.$ Let $r >0$ be some positive number: later we will choose $r$ to depend on $\rho_0$, but at first we
simply require $r$ to be some positive radius.
Let $y_0 := \ga(\rho_0 - (r/2) )$. In particular $d(x_0,y_0) = d(\ga(\rho_0), \ga(\rho_0 - (r/2) ) ) = 
r/2$.
Due to the triangle inequality we have:
$d(x,x_0) \leq d(x,y_0) + d(y_0,x_0) =  d(x,y_0) + (r/2)$ for all $x \in M$ and hence
\begin{eqnarray*}
B_{\de r}(y_0) \subset B_{r}(x_0).
\end{eqnarray*}
Furthermore we have $d(x,b) \leq d(x,y_0) + d(y_0,b) = d(x,y_0) + \rho_0 - (r/2)$
for all $x \in M$ and hence
\begin{eqnarray*}
B_{\de r}(y_0) \subset B_{\rho_0}(b).
\end{eqnarray*}
In particular, using this inclusion and the fact that $\phi$ is non-decreasing, we have
\begin{eqnarray}
\phi(\rho(x)) \leq \phi(\rho_0) \ \forall x \in B_{\de r}(y_0) \label{goody}.
\end{eqnarray}
Now set $r:= e^{-(1/2)\phi(\rho_0)}.$
We obtain a lower bound for $\phi(\rho(x))  - \phi(\rho_0)$ for $x \in 
B_{\de r}(y_0)$ as follows.

First note that for all $x \in B_{\de r}(y_0)$ we have by the triangle inequality
\begin{eqnarray}
\rho(x) = d(x,b) &&\geq d(b,y_0) - d(x,y_0) \cr
&&= \rho_0 -(r/2) -d(x,y_0) \cr
&&\geq \rho_0 - (r/2) -\de r \cr
&&= \rho_0 -r(1/2 + \de). \nonum \\
\label{rholower1}
\end{eqnarray}
From the mean value theorem and the fact that $\rho(x) \leq  \rho_0$
for $x \in B_{\de r}(y_0)$, we get
\begin{eqnarray*}
|\phi(\rho(x)) -  \phi(\rho_0)| &&=  |\phi'(z_x)||\rho_0 - \rho(x)|
\end{eqnarray*}
for some $z_x \in [\rho(x),\rho_o]$.
From the above (\eqref{rholower1}), we have that $\rho(x) \geq \rho_0 -r(1/2 + \de)$.
Hence,
\begin{eqnarray*}
|\phi(\rho(x)) -  \phi(\rho_0)| &&\leq  |\phi'(z_x)|r (1/2 + \de).
\end{eqnarray*}
Using the third estimate of \eqref{thehs2}, and the fact that $\phi$ is non-decreasing, we get
\begin{eqnarray}
|\phi(\rho(x)) -  \phi(\rho_0)| &&\leq  c |e^{(1/2)\phi(z_x)}| r (1/2 + \de) \cr
&& \leq c |e^{(1/2)\phi(\rho_0)}| r (1/2 + \de) \cr
&& = c(1/2 + \de), \nonum \\
\label{phirhoest}
\end{eqnarray}
in view of the definition of $r$.

Define $\ti g(x): = e^{\phi( \rho(x))} g(x)$.
Balls with respect to $\ti g$ will be denoted with a tilde:
$\ti B_s(p)$ is the ball with radius $s$ and centre $p \in M$ with respect to $\ti g$.
We denote distance with respect to $\ti g$ also with a tilde:
$\ti d(x,y) $ is the distance with respect to $\ti g$ from $x$ to $y$.
The volume form with respect to $g$ is denoted by $d\mu_g$, and that of $\ti g$ with
 $d\mu_{\ti g}$.
We wish to show that $(M,\ti g)$ is also non-collapsed.
Let $x \in M$ be given. Now set $r = e^{-(1/2)\phi(\rho_0)}.$

{\bf Claim:    $B_{\de r}(y_0) \subset \ti B_{2}(x_0) $.}
Let $\si:[0,l] \to B_{\de r}(y_0) $  be a legth minimizing 
geodesic of unit speed with respect to $g$, $l \leq \de r$, $\si(0) = y_0$,
$x = \si(l) \in B_{\de r}(y_0)$.
Then 
\begin{eqnarray*}
\ti d(y_0,x) && = 
\ti d (y_0,\si(l))\cr
  &&\leq \int_{0}^l \sqrt{\ti g (\si'(s), \si'(s))} ds  \cr
&&= \int_{0}^l e^{(1/2)\phi(\rho(\si(s)))} \sqrt{g( \si'(s), \si'(s)) }ds\cr
&&\leq e^{(1/2)\phi(\rho_0)} \int_{0}^l \sqrt{g( \si'(s), \si'(s)) }ds\cr
&&= e^{(1/2)\phi(\rho_0)} l\cr
&& \leq e^{(1/2)\phi(\rho_0)} \de r \cr
&& =  e^{(1/2)\phi(\rho_0)}  \de e^{-(1/2)\phi(\rho_0)} \cr
&&=\de \leq 1,
\end{eqnarray*}
in view of equation \eqref{goody}, the definition of $r$ and the fact that $\si$ is distance minimizing
(w.r.t. $g$).
Furthermore 
\begin{eqnarray*}
\ti d(y_0,x_0) && \leq \int_{\rho_0 - (r/2)}^{\rho_0} \sqrt{\ti g (\ga'(s), \ga'(s))} ds  \cr
&&=  \int_{\rho_0 - (r/2)}^{\rho_0}  e^{(1/2)\phi(\rho(\ga(s)))} \sqrt{g( \ga'(s), \ga'(s)) }ds\cr
&&\leq e^{(1/2)\phi(\rho_0)} \int_{\rho_0 - (r/2)}^{\rho_0} \sqrt{g( \ga'(s), \ga'(s)) }ds\cr
&&= e^{(1/2)\phi(\rho_0)} (r/2)\cr
&& = (1/2) e^{(1/2)\phi(\rho_0)}  e^{-(1/2)\phi(\rho_0)} \cr
&&= 1/2,
\end{eqnarray*}
in view of the definition of $r$ and the fact that $\ga$ is distance minimizing.
Hence
\begin{eqnarray*}
\ti d(x_0,x) && \leq \ti d(x_0,y_0) + \ti d(x,y_0) \cr
&& \leq 3/2
\end{eqnarray*}
for all $x \in B_{\de r}(y_0)$, which proves the claim.

This means that  
\begin{eqnarray*}
\ti \vol(\ti B_2( x_0)) && = \int_{ \ti B_{2}(x_0)} d\mu_{\ti g}(x)\cr 
&& \geq \int_{B_{\de r}(y_0) }   d\mu_{\ti g}(x)\cr
&& =  \int_{B_{\de r}(y_0) } e^{(n/2)\phi(\rho(x))} d\mu_g(x) \cr
&& \geq  \int_{B_{\de r}(y_0) } e^{(n/2)(\phi(\rho_0) - c(1/2 + \de) ) } d\mu_g(x) \cr
&& = e^{-(n/2)c(1/2 + \de) } e^{(n/2)(\phi(\rho_0))} \int_{B_{\de r}(y_0) } d\mu_g(x) 
\end{eqnarray*}
in view of the claim and \eqref{phirhoest} .
Hence
\begin{eqnarray}
 \ti \vol(\ti B_2( x_0)) && \geq \ti c e^{(n/2)\phi(\rho_0)} \int_{B_{\de r}(y_0) } d\mu_g(x)\cr
&& \geq \ti c e^{(n/2)\phi(\rho_0)} r^n \de^n  v_0 = \ti c v_0 \de^n =: \ti v_0, \label{volumy}
\end{eqnarray}
in view of the non-collapsed condition and the definition of $r = e^{-(1/2)\phi(\rho_0)}$.
Note: $\ti c = e^{-(n/2)c(1/2 + \de) } >0$ is a universal constant which depends only on 
$n$ and $h$ (the exponential comparison function which was used to define $\phi$).

The well known formulas for the change of the metric $g$ to $\ti g = e^{f} g = \psi g$
(for example see equation 13 in Chapter 8 of \cite{Sh}) 
for a function $f:M \to \R$ (where here $\psi$ is defined to be
$\psi(x) := e^f(x)$)
 are 
\begin{eqnarray}
  \ti \Ricci_{ij} =&& \Ricci_{ij} -{(n-2)}{ 2} (\grad^2 f)_{ij} + \frac{n-2}{4} \grad_i f   \grad_j f  
- {1 \on 2} ( \lap f -\frac{n-2}{2} {{}^{g} |}\grad f|^2) g_{ij}, \cr
\ti \Riem_{ijkl}  = && \psi \Riem_{ijkl} + {1 \on 2}(g_{jk} (\grad^2 \psi)_{il}
- g_{jl} (\grad^2 \psi)_{ik} - g_{ik}  (\grad^2 \psi)_{jl} + g_{il}  (\grad^2 \psi)_{jk})\cr
&& +{3 \on 4 \psi}( \ \ \ g_{ik} \grad_j \psi \grad_l \psi - g_{jk} \grad_i \psi \grad_l \psi \cr
&& \ \ \  \ \ \  \ \  +  \  g_{jl} \grad_i \psi \grad_k \psi - g_{il} \grad_j \psi \grad_k \psi) \cr
&&+ {1 \on 4 \psi}( g_{jk} g_{il} - g_{ik} g_{jl} )g^{pq} \grad_p \psi 
\grad_q \psi, \nonum \\
\label{shicurv}
\end{eqnarray} 
where $\grad l$ denotes the gradient of the function $l$, and $(\grad^2 l)$ denotes the second covariant derivative of $f$ (which is a $(^0_2)$ Tensor), 
both w.r.t to $g$.
In the following ${{}^g|} \cdot |$ will denote the norm with respect to $g$. 
Now let $f$ be
$f(x) = \phi(\rho(x))$ (this implies $\psi (x) = e^{f(x)} = e^{\phi(\rho(x))}$ ) where 
$\rho$ is the distance  function with respect to $g$, and 
$\phi: \R \to \R$ is an arbitrary  smooth function.
Our assumption of {\it controlled geometry at infinity}
implies that
$^g|\grad \rho | = 1$ on $M-B_R(b)$  
and one version of the Hessian comparison principle tells us that 
\begin{eqnarray*}
{}^{g}|\grad^2 \rho|(x) && \leq  c \rho(x)( R_B(\rho(x)) +c(n)(k + 1)),
\end{eqnarray*} 
wherever $\rho$ is differentiable and larger than one (see Appendix A )
and here $R_B: \R^+ \to \R$ is the function
\begin{eqnarray}
R_B(r) := r  (\sup_{B_r(x_0)} |\Riem(g)|) e^{(n \sup_{B_r(x_0)}
|\Riem(g)| + 1)r}. \label{RBr}
\end{eqnarray}

The following identities then follow from the definitions of 
$f(x) = \phi(\rho(x))$,  and $\psi = e^{\phi(\rho)} = e^f$:

\begin{eqnarray}
\grad_i f (x) = &&  \phi'(\rho(x)) \grad_i\rho(x), \cr
\grad_i  \psi (x) = &&   e^f(x)( \grad_i f)(x) =   \psi(x)  \phi'(\rho(x))
\grad_i\rho(x), \cr
(\grad^2 f)_{ij}(x) =  &&  \phi''(\rho(x)) \grad_i\rho(x)  \grad_j \rho(x)  +  \phi'(\rho(x))(\grad^2 \rho)_{ij}(x),\cr
(\grad^2 \psi)_{ij}(x) =  &&  \psi(x)  |\phi'(\rho(x))|^2  \grad_i\rho(x) \grad_j \rho(x) + \psi(x)  \phi''(\rho(x)) \grad_j\rho(x)\grad_i\rho(x)  \cr
&& + \psi(x)  \phi'(\rho(x)) (\grad^2 \rho)_{ij}(x). \nonum 
\\ \label{gradeq}
\end{eqnarray}
Assume that $\phi$ satisfies 
\begin{eqnarray}
|\phi'|&&\leq c e^{\phi/8} \cr 
|\phi'|^2&& \leq c e^{\phi/4} \cr
|\phi''|&& \leq c e^{\phi/8} \nonum \\
\label{thephis}
\end{eqnarray}
for some universal constant $c$ not depending on $k$
and $n$ (later we will examine different $\phi$'s but they all satisfy an estimate
of the form above for the same constant $c$).
Using $|\grad \rho |^2 = 1,$ \eqref{gradeq}, \eqref{thephis} and that $\rho$
is $k$- concave, we get
\begin{eqnarray*}
^g|\grad f| && \leq ce^{f/8},  \cr
(\grad^2 f) && \leq c(n,k)e^{f/8} g.
\end{eqnarray*}

Hence 
\begin{eqnarray}
  \ti \Ricci_{ij} = && \Ricci_{ij} -{(n-2)}{ 2} (\grad^2 f)_{ij} + \frac{n-2}{4} \grad_i f   \grad_j f  
- {1 \on 2} ( \lap f -\frac{n-2}{2} {{}^{g} |}\grad f|^2) g_{ij} \cr
&& \geq -|k| g_{ij} - c(n,k)e^{f/3} g_{ij}   \cr
&& \geq -c(n,k) \ti g_{ij}, \nonum \\
\label{riccity}
\end{eqnarray}
since $g_{ij} = e^{-f}{\ti g}_{ij}$ and $f>0$.

We will assume in the following that
\begin{eqnarray}
\lim_{r \to \infty} e^{-(1/8)\phi(r)} R_B(r+2) = 0, \label{assumpty}
\end{eqnarray}
where $R_B(r)$ is the function introduced above in \ref{RBr}. 
We estimate the equalities \eqref{gradeq} using the growth properties of $\phi$
(\eqref{thephis})
as follows
\begin{eqnarray}
{}^g|\grad \psi (x))|^2 \leq &&   c\psi^{9/4},  \cr
{}^g |\grad^2 \psi|(x) \leq  && 
 \psi^{5/4}(x)  c(\rho (x)+2) ( R_B(\rho(x) + 2) +c(n)(k + 1)).  \nonum \\ \label{gradeq2}
\end{eqnarray}
%and hence (use also that $\psi \geq 1$)
%\begin{eqnarray}
%{}^{\ti g}|\grad \psi (x))|^2 \leq &&   c\psi^{5/4)}(x)  \cr
%{}^{\ti g} |\grad^2 \psi|(x) \leq  && 
% c \psi^{+1/4}(x)  \rho (x) ( R_B(\rho(x)) +c(n)(k + 1))   \label{gradeq3}
%\end{eqnarray}
%and hence
%\begin{eqnarray}
%{}^{\ti g}|\grad \psi (x))| \to 0  \ \ \mbox{ as } \rho(x) \to \infty  \ 
%{}^{\ti g} |\grad^2 \psi|(x) \leq  &&  c   
% \psi^{-1/4}(x)  \rho (x) ( R_B(\rho(x)) +c(n)(k + 1))   \to 0  \ \ \mbox{ as } \rho(x) \to \infty  \  \label{gradeq3}
%\end{eqnarray}
%in view of assumption \eqref{assumpty} and the fact that $\psi(x) = e^{\phi(\rho(x))}$.
Hence, using formula \eqref {shicurv}, we get
\begin{eqnarray}
^{\ti g}|\ti \Riem|  \leq && {1 \on \psi} |\Riem| + c \psi^{-2} \psi^{5/4}(x)  (\rho (x)+2) ( R_B(\rho(x)+2) +c(n)(k + 1)) \cr
 && + c \psi^{-3} (  \psi^{9/4} ) \cr
&& \leq  {1 \on \psi} |\Riem|  + c \psi^{-3/4}  (\rho (x)+2) ( R_B(\rho(x) + 2) +c(n)(k + 1))\cr
&& \to 0 \ \  \mbox{ as } \ \ \rho(x) \to \infty, \nonum \\
\label{curvzero}
\end{eqnarray} 
in view of
\eqref{assumpty} and the fact that $\psi(x) = e^{\phi(\rho(x))}$.

Choose $\phi = \phi_i ,$ where $\phi_i(r):= h((r-i)_+^4)$ and 
$h$ is an exponential
comparison function such that  
\begin{eqnarray*}
\lim_{r \to \infty} e^{-(1/8)h(r)} R_B(r+2) = 0. 
\end{eqnarray*}
Then trivially 
\begin{eqnarray}
\lim_{r \to \infty} e^{-(1/8)\phi_i(r)} R_B(r+2) = 0. \label{assumpty2}
\end{eqnarray}
Note that $\phi_i$ satisfies
\begin{eqnarray}
|\phi_i'|&&\leq c e^{\phi/8} \cr 
|\phi_i'|^2&&\leq c e^{\phi/4} \cr
|\phi_i''|&&\leq c e^{\phi/8}, \nonum \\ 
\label{thephis2}
\end{eqnarray}
as demanded in \eqref{thephis},
in view of Lemma \ref{littely}.
That is, $\phi = \phi_i$ satisfies all the required conditions of this section.
This, \ref{volumy}, \ref{riccity} and \ref{curvzero} imply that
$g_i(x) := e^{\phi_i(\rho(x))}g(x)$ is a metric satisfying
\begin{eqnarray*}
&&g_i = g \ \mbox{ for all } x \in {B_{(i/2)}(p_0)}, \cr
&&\Ricci(g_i) \geq -c(n,k) g_i,  \cr
&&\vol(B_1(x_0),g_i) \geq \ti v_0,  \  \mbox{ for all } x  \in M \cr 
&&\sup_M {{}^{g_i} |}\Riem(g_i)| < \infty,
\end{eqnarray*}
as required.
\end{proof}

\section{ Applications}

Let $(M,g_0) \in \curlT(3,k,m,v_0)$ and let $(M,\upi g_0 ) \in 
\curlTi(3,k,\ti v_0)$ be the smooth metrics constructed in the previous 
section : remember that these $\upi g_0$ satisfy
$\upi g_0 = g_0 \ \mbox{ for all } x \in {B_i(p_0)} $.

Now we may apply Theorem \ref{maintheo} to each $(M,\upi g_0 )$ to obtain 
solutions $(M,g_i(t))_{t \in [0,T(n,\ti v_0))}\in \curlTi(3,\ti k, \ti v_0)$ 
satisfying the a priori estimates. Hence using the local estimates of Theorem 1.3 of \cite{Si5}
and the interior estimates of Shi (see \cite{CCCY}), 
we may take a Hamilton limit to get a solution to Ricci flow 
$(M,g(t)_{t \in [0,T)})$ which satisfies the a priori estimates \eqref{theimes}.
Note that the local estimates of Theorem 1.3 in \cite{Si5} 
guarantee that we may take the limit on the interval $[0,T)$ and
not just $(0,T)$.
So we have proved:

\begin{theo}\label{maintheo2}
Let $(M, g_0)$ be a three (or two) manifold in
$\curlT(3,k,m,v_0)$ ($\curlT(2,k,m,v_0)$).
Then there exists a $T = T(v_0,k,m)>0$ and a solution
$(M,g(t))_{t \in [0,T)}$ to Ricci flow, 
satisfying  \eqref{theimes}.
\end{theo}

In a more general setting we prove the following.

\begin{theo}\label{realmaintheo2}
Let $(M_i, \upi g_0)$ be a sequence of three (or two) manifolds in
$\curlT(3,k,m,v_0)$ ($\curlT(2,k,m,v_0)$) and let
$(X,d_X,x) = \lim_{i \to \infty} (M_i, d(\upi g_0), x_i)$ 
be a pointed Gromov-Hausdorff limit of this sequence.
Let $(M_i,\upig(t))_{t \in [0,T)}$ be the solutions to Ricci-flow 
coming from the theorem above. Then (after taking a sub-sequence if necessary)
there exists a Hamilton limit solution $(M,g(t),y)_{t \in (0,T)} := \lim_{i \to \infty}
(M_i,\upig(t),x_i)_{t \in (0,T)}$ satisfying  \eqref{theimes} and
\begin{itemize}
\item[(i)] $(M,d(g(t)),y) \to (X,d_X,x)$ in the Gromov-Hausdorff sense
as $t \to 0$.
\item[(ii)] $M$ is diffeomorphic to $X$. In particular, $X$ is a manifold.
\end{itemize}

\end{theo}

\begin{proof}

We apply the Theorem \ref{maintheo2} to obtain (after taking a subsequence if necessary) 
a limit solution
$(M,g(t),y)_{t \in (0,T)} := \lim_{i \to \infty}
(M_i,\upig(t),x_i)_{t \in (0,T)}$ satisfying  the estimates \eqref{theimes}.
We prove that $(M,d(g(t)),y) \to (X,d_X,x)$ as $t \to 0$ as follows.
We introduce the notation $d(t) = d(g(t))$ and $d_i(t) = d(g_i(t))$.
In view of the Lemma \ref{diamlemm} 
we have $\GHd((B_r(x_i),d_i(t)),(B_r(x_i),d_i(0))) \leq c(r,t)$
where $c(r,t) \to 0 $ as $t \to 0$ and $c(r,t)$ does not depend on $i$.
Furthermore $\GHd( (B_r(x_i),d_i(0)), (B_r(x),d_X))  \leq l(i,r)$
where $l(i,r) \to 0$ as $i \to \infty$,
and  $\GHd( (B_r(x_i),d_i(t)), (B_r(y),d(t)) )  \leq s(i,r,t)$
where $s(i,r,t) \to 0$ as $i \to \infty,$
in view of the fact that $(M,d_i(t),x_i) \to (M,d(t),y)$ and
$(M_i,d_i(0),x_i) \to (X,d_X,x)$ in GH sense
as $i \to \infty$.
Hence, since Gromov Hausdorff distance satisfies the triangle inequality, we obtain 
(for $r$ fixed):
\begin{eqnarray*}
 &&\GHd((B_r(y),d(t)),(B_r(x),d_X))  \cr
 &&\leq \GHd((B_r(y),d(t)), (B_r(x_i),d_i(t)) ) + 
\GHd((B_r(x),d_X),(B_r(x_i),d_i(t)))  \cr
&&\leq  \GHd((B_r(y),d(t)), (B_r(x_i),d_i(t)) )  
+ \GHd((B_r(x),d_X),(B_r(x_i),d_i(0)))\cr
&& \ \ \ +  \GHd(  (B_r(x_i),d_i(t)),(B_r(x_i),d_i(0))  ) \cr
&& \leq  s(i,r,t) + l(i,r) +  c(r,t).
\end{eqnarray*}
Letting $i \to \infty$ ($t$ and $r$ fixed), 
we get
\begin{eqnarray*}
\GHd((B_r(y),d(t)),(B_r(x),d_X)) \leq c(r,t),
\end{eqnarray*}
in view of the properties of $ s(i,r,t)$ and $l(i,r)$, and hence
$(M,d(g(t)),y) \to (X,d_X,x)$ as $t \to 0$ since
$c(r,t) \to 0$ as $t \to 0$.

Finally we show that $(M,d(t),y)$ is diffeomorphic to $(X,d,x)$ for all $t \in (0,T)$.
The limit solution satisfies the estimates \eqref{theimes}. 
So $d(t_i)(p,q)$ is a Cauchy sequence in $i$ for any sequence $t_i \to 0$.
In particular, we obtain a limit as $i \to \infty$:
let us call this limit $l(p,q)$.
Clearly $l(p,q)$ does not depend on the sequence
$t_i$ we choose.
$l(\cdot,\cdot)$ satisfies the triangle inequality, as $d(t)(\cdot,\cdot)$ does,
for all $t>0$. Also
$l(p,p) = \lim_{t \to 0} d(p,p,t) = 0.$
Furthermore: $l(p,q) >0$ for all $p \neq q$:
\begin{eqnarray}
l(p,q) &&= \lim_{s \to 0}d(s)(p,q)\cr
&&\geq  \lim_{s \to 0}e^{c_1(c_0,n)(s-1)}  d(1)(p,q) \cr
&&= e^{-c_1(c_0,n)}  d(1)(p,q) >0. \label{dissy} 
\end{eqnarray}
That is, $l$ is a metric.
From the above estimates \eqref{dissy}, we see that 
$d(t)(\cdot,\cdot) \to 
l(\cdot,\cdot)$ as $t \to 0$ uniformly on compact sets $K \subset M$ (compact with respect 
to $d(t)$ for any $t$) .
This implies that $(M,d(t),y) \to (M,l,y)$ as $t \to 0$ in the $C^0$ sense on compact sets.
 
Now we show that the metric $l$ defined on the set $M$ defines the same
topology as that of $(M,d(t))$ for any $t$.
First note that all of the $(M,d(t))$ for $t>0$ have the same topology:
$(M,g(t))$ are smooth Riemannian metrics with bounded curvature evolving
by Ricci flow and are all equivalent. Let us denote this topology by
$\curlO$.
We denote the topology coming from $(M,l)$ by
$\ti \curlO$.

We use the notation ${\upl B}_r(x)$ to denote a ball of radius $r>0$
at $x \in M$ with respect to the metric $l$, and (as usual)
 ${\up{d(t)}  B}_r(x)$ to denote a Ball of radius $r>0$ at $x \in M$
with respect to the metric $d(t)$.
From the above inequalities and the definition of $l$ we have
\begin{eqnarray}\label{openest}
{ \up{d(t)} B}_{(r - c_2\sqrt{t})}(p) \subset {\upl B}_{r}(p)
\subset {\up{d(t)}  B}_{re^{c_1t}}(p) 
\end{eqnarray}
for all $p \in M$.

For any open set $U$ in $\ti \curlO$ we therefore have
\begin{eqnarray*}
U && = \cup_{p \in U}{\upl B}_{r(p)}(p) \cr
&& = \cup_{p \in U}{\up{d(t(p))} B}_{(r(p)-c_2 \sqrt{t(p)})} (p), \cr
\end{eqnarray*}
where $r(p)>0$ is chosen small so that
${\upl B}_{r(p)}(p) \subset U$ and
$t(p)>0$ is chosen small so that $ r(p)-c_2 \sqrt{t(p)} >0$.
Hence,
$U$ is in $\curlO$.
Now assume $V \in \curlO$. 
Then, using the estimate \ref{openest} again, we see that
\begin{eqnarray*}
V &&= \cup_{p \in V}{{}^{d(t)} B}_{r(p,t)}(p) \cr
&& = \cup_{p \in V}{\upl B}_{r(p,t)e^{-c_1 t} }(p), \cr
\end{eqnarray*}
where $r(p,t)$ is chosen small so that
${ {}^{d(t)}B}_{r(p,t)}(p) \subset V.$ 
Hence  $V \in \ti \curlO$. 
Hence, the identity from $(M,l,y)$ to $(M,d(t),y)$ is
a homeomorphism. 
We already showed that $(M,d(t),y) \to (X,d_X,x)$ as $t \to 0$ 
($(X,d_X,x)$ was defined by $(X,d_X,x) := \lim_{i \to \infty} (M_i, d_i(0), x_i)$).
Hence $(X,d_X,x) = (M,l,y)$, and $(X,d_X,x)$ is homeomorphic to $(M,d(t),y)$.
In three dimensions every manifold has a unique smooth maximal structure.
This finishes the proof.
\end{proof}

We formulate the last result of the theorem  above in a form independent of the
Ricci flow.
\begin{prop}
Let $(M_i,g_i)$ be a sequence of smooth 3-manifolds (2-manifolds)
in $\curlT(3,k,m,v_0)$ ($\curlT(2,k,m,v_0)$) and $(X,d_X,x)$ be a Gromov-Hausdorff limit
of $(M_i,d(g_i),x_i)$ (such a  $(X,d_X,x)$ always exists after taking a subsequence). Then 
\begin{itemize}
\item $(X,d_X)$ is a manifold.
\item if $\diam(M_i,g_i) \leq d_0 < \infty$ for all $i \in N$, then
$M_i $ is diffeomorphic to $X$ for $i \in \N $ sufficiently large.
\end{itemize}
\end{prop}

As a corollary to this result and Theorem \ref{realmaintheo2} and Lemma \ref{seclemma2}
we obtain the following corollary.

\begin{coro}\label{semicoro}
Let $(M_i, {\upi g}_0)$,$i \in \N$  be a sequence of three (or two) manifolds
with $ (M_i, {\upi g}_0) \in \curlT(3,-{1 \on i},m,v_0)$ 
($\curlT(2,-{1 \on i},m,v_0))$
for each 
$i \in \N$: note this implies
$$\Ricci(M_i,{\upi g}_0) \geq - \frac 1 i .$$
Let $(X,d_X) = \GHlim_{i \to \infty}(M_i,d(  {\upi g}_0) ).$
Then the solution 
$(M,g(t),x)_{t \in (0,T)}$ obtained in Theorem \ref{maintheo2}
satisfies $$ \Ricci(g(t)) \geq 0$$ for all $t \in (0,T)$ and $(X,d_X)$ is
diffeomorphic to $(M,g(t))$  for all $t \in (0,T)$.
In particular, combining this with the results of Shi \cite{Sh2}  and Hamilton \cite{Ha82}, 
we get that $(X,d_X)$ is diffeomorphic to  $\R^3$, $\Sphere^2 \times \R$ or $\Sphere^3$ 
modulo a  group of fixed point free isometries in the
standard metric.
\end{coro}

\section{Hessian comparison principles}

Let $\rho:\hat M := M  - \cut(p) \to \R$ be the
distance function from some fixed $p$, $\rho(x) := \dist(p,x)$, and let
$q \in \hat M$.
Let $\ga:[0,l] \to M$ be the unique minimizing 
geodesic from $p$ to $q$ with $|\ga'(t)| =1$ for all $t \in [0,l]$.
We denote the set of smooth vector fields along $\ga$ by $T_\ga M$: $V \in T_\ga M$ means 
$V:[0,l] \to TM$ is smooth with $V(s) \in T_{\ga(s)} M$ for all $s \in [0,l]$.
$V': [0,l] \to TM$ will denote the vectorfield along $\ga$
obtained by taking the {\it covariant derivative of $V$ along $\ga$}: see the book of
do Carmo \cite{Do} for an explanation.
Let $X_q \in T_q M$ be normal to $\ga$.
It is well known (see \cite{ScYa} chapter 1) that
$\rho$ is differentiable on $\hat M$, 
and that $\grad \rho(q) = \ga'(l)$ and 
\begin{eqnarray}
\grad^2 \rho(q)(X_q,X_q) = \int_0^l {   {}^g |} \ti X'(s)|^2  - \Riem(g)(\ti X,\ga',\ti X,\ga') ds, 
\label{gradsq}
\end{eqnarray}
where $\ti X \in T_\ga M$ is the unique Jacobi field along $\ga$ such that
$\ti X(0) = 0$ and $\ti X(l) = X_q$ (see the Book of do Carmo \cite{Do} for a discussion on
Jacobi fields).

The tensor inequality 
$$\grad^2 \rho  \leq c(n,k)  g $$ in the case that the sectional curvatures are
bounded from below by $k$ is well known: a proof  may be found
 in (for example) \cite{ScYa}, chapter 1.
Here we show how to obtain a more general inequality which bounds
$\grad^2 \rho$ from above and below, for constants which depend on the 
suprmemum of the curvatures in a geodesic Ball fo radius $r$ where
we are evaluating $\grad^2 \rho$.

NOTE: To be consistent with the rest of this paper
I am using the convention that sectional curvatures of a plane spanned by two perpendicular
vectors $v,w$ of length one is $\sec(v,w) = \Riem(v,w,v,w)$ 
and that the sectional curvature on the sphere is positive
(in \cite{ScYa}   $\sec(v,w) = \Riem(v,w,w,v)>0$ on the sphere). 
$\ti X$ is a Jacobi field means then that
$\ti X'' - \Riem(\ti X,\ga',\ga') = 0$.
Let $E_i \in T \ga M,$ $i =1, \ldots, n$ be parallel fields ($E_i'=0$) such that $\{ E_i(t) \}_{i = 1}^n$ is an orthonormal basis at 
$\ga(t)$ for each $t \in [0,l]$.
Let $f_i(s) := g(\ti X(s),E_i(s))$.
Let $k:= \sup \{ |\Riem(\ga(s))| | s \in [0,l] \}.$
Then the Jacobi field equation implies 
\begin{eqnarray*} 
f_i''(s) &&= g(\ti X''(s),E_i(s)) \cr
&& = \Riem(\ti X,\ga',\ga',E_i(s)) \cr
&& = \sum_{j = 1}^n f_j \Riem(E_j,\ga',\ga',E_i) \cr
\end{eqnarray*}
and hence $f(s):= |\ti X(s)|^2 =  \sum_{i=1}^n (f_i)^2(s)$ satisfies
\begin{eqnarray*} 
f'' && =   \sum_{i,j = 1}^n 2f_i f_j \Riem(E_j,\ga',\ga',E_i) + \sum_{i = 1}^n2((f_i)')^2 
\cr
&& \geq -k n \sum_{i,j = 1}^n (f_i(s))^2 \cr
&&= -kn f(s).
\end{eqnarray*}
This implies that
$g(s) = e^{cs}f(s)$ satisfies
\begin{eqnarray*} 
g''(s) && = e^{cs}f''(s)+ ce^{cs}f'(s) + c^2e^{cs}f(s) \cr
&& \geq ( - kn +c^2)e^{cs} f(s) +   ce^{cs}f'(s)  \cr
&& =    ( - kn +c^2)e^{cs} f(s) + g'(s) - ce^{cs}f(s)  \cr
&& = ( - kn +c^2 -c)e^{cs} f(s) +g'(s)  \cr
&& > g'(s) 
\end{eqnarray*}
if for example $c = kn + 1$.
Hence $g$ has no local maximum in $(0,l)$: if it did, we would obtain
\begin{eqnarray*} 
0 \geq g''(s) > 0
\end{eqnarray*}
which is a contradiction.
Now note that $g(0)= 0$ and 
$$g(l) = e^{(kn + 1)l}f(l) = e^{(kn + 1)l}|\ti X(l)|^2 =     e^{(kn + 1)l}$$
since $\ti X(0) = 0$ and $\ti X(l) = X(q)$ and $|X(q)| = 1$.
This implies that $g(s) \leq e^{(kn + 1)l}$ for all $s \in [0,l]$ and hence that
$f(s) = |\ti X(s)|^2 \leq e^{(kn + 1)l}$ for all $s \in [0,l]$.
Let
$R_B(r) :=   r  \sup_{B_r(x_0)} |\Riem(g)| e^{(n \sup_{B_r(x_0)} |\Riem(g)| + 1)r}$.
Then, using \eqref{gradsq}, we get
\begin{eqnarray*}
\grad^2 \rho(q)(X_q,X_q) && = \int_0^l {   {}^g |} \ti X'(s)|^2  - \Riem(g)(\ti X,\ga',\ti X,\ga') ds  \cr
&& \geq    \int_0^l     -k |\ti X|^2 ds  \cr
&&\geq -kl  e^{(kn + 1)l}  \cr
&& \geq -R_B(l)
\end{eqnarray*}
for every $q \in B_r(p) \cap \hat M$ as required.
The estimate
$\grad^2 \rho(q) \leq R_B(l)$ follows by using the standard hessian comparison
principle (see Chapter 1 \cite{ScYa}), and the fact that $\sec(x) \geq - \sup_{B_l(x_0)} |\Riem(x_0)| $
for $x \in B_l(x_0))$.

\section{Estimates on the distance function for Riemannian manifolds evolving by Ricci flow}

For completeness, we prove some results which are implied or proved in \cite{Ha95} 
and stated in \cite{CCCY} as editors' note 24 from the same paper in that book.
The lemma we wish to prove is 
\begin{lemm}\label{diamlemm2}
Let $(M^n,g(t))_{t \in [0,T)}$ be a solution to Ricci flow
with \begin{eqnarray*}
&&\Ricci(g(t)) \geq -c_0, \cr
&& | \Riem(g(t))| t \leq c_0 
\end{eqnarray*}
Then
\begin{eqnarray} 
e^{c_1(c_0,n) (t-s)} d(p,q,s) \geq  d(p,q,t) \geq  d(p,q,s) -c_2(n,c_0) (\sqrt t  - \sqrt s)\label{diogothree}
\end{eqnarray}
for all $ 0 \leq s \leq t \in [0,T)$.
\end{lemm} 

\begin{proof}
The  inequality 
$$ d(p,q,t) \geq  d(p,q,0) -c_1(n,c_0) \sqrt t$$
is proved in \cite{Ha95}, theorem 17.2 after making a slight modification of the proof.
If we examine the proof there (as  pointed out in \cite{CCCY} as editors note 24 of the same book),
we see in fact that what is proved is:
$$d(P,Q,t) \geq d(P,Q,s) -C \int_s^t \sqrt{M(t)},$$
where $\sqrt{M(t)}$ is any integrable function which satisfies
$$\sup_M |\Riem(\cdot,t)| \leq M(t).$$
In particular, in our case we may set $$M(t) = {c_o \on t},$$ which then implies the inequality
$d(p,q,t) \geq  d(p,q,s) -c_2(n,c_0) (\sqrt t  - \sqrt s)$.
The second inequality is also a simple consequence of results obtained in \cite{Ha95}.
Lemma 17.3 tells us that
$$ \partt d(P,Q,t) \leq -\inf_{\ga \in \Gamma} \int_{\ga} \Ricci(T,T)ds,$$
where the $\inf$ is taken over the compact set $\Gamma$ of all geodesics from $P$ to $Q$ realising the distance
as a minimal length, $T$ is the unit vector field tangent to $\ga.$
Then in our case $\Ricci \geq -c_0$ implies
$$ \partt d(P,Q,t) \leq c_0 d(P,Q,t).$$
This implies that 
$$ d(P,Q,t) \leq  \exp^{c_0 (t-s)}  d(P,Q,s),$$ 
as required.
\end{proof}

\vskip 0.1 true in
\section*{Acknowledgements\markboth{Acknowledgements}{Acknowledgements}}

We would like to thank Bernhard Leeb, Lorenz Schwachh\"ofer,
Burkhard Wilking and Christoph B\"ohm 
for their interest in and comments on this work. 
We thank the referee for pointing
out in incorrect use of the hessian comparsion principle
(and some related issues) in an earlier version of this work.

\vskip 0.1 true in
\rm

\vskip 0.3 true in
\centerline{Mathematisches Institut, Eckerstr. 1, 79104 Freiburg im Br., Germany }
 \centerline{e-mail: msimon@mathematik.uni-freiburg.de}
\end{document}